\RequirePackage{fix-cm}
\documentclass[smallextended]{mysvjour3}       
\smartqed  

\usepackage{thumbpdf}
\usepackage[%
  colorlinks%
  ,bookmarks={false}%
  ,pdftitle={}%
  ,pdfsubject={}%
  ,pdfkeywords={}%
  ,pdfauthor={Tom Maerz <maerz@maths.ox.ac.uk>}%
  ]{hyperref}

\usepackage{a4wide}
\usepackage{caption}
\usepackage{amsmath}
\usepackage{amssymb}
\usepackage{graphicx}
\usepackage{enumerate}
\usepackage{subfigure}
\usepackage{tikz}
\usetikzlibrary{matrix,arrows}
\usepackage[latin1]{inputenc}
\usepackage{slashbox}
\usepackage{algorithm}
\usepackage{algorithmicx}
\usepackage{algpseudocode}

\newcommand{\T}{\vec{T}}
\newcommand{\Proj}{\tens{P}}
\newcommand{\No}{\tens{N}}
\newcommand{\x}{\vec{x}}
\newcommand{\vphi}{\vec{\varphi}}
\newcommand{\Defo}{\tens{H}}

\newcommand{\N}{{\mathbb  N}}
\newcommand{\R}{{\mathbb  R}}
\newcommand{\Z}{{\mathbb  Z}}

\newcommand{\PS}{{\mathbb  P}}
\newcommand{\laplace}{\Delta}

\newcommand{\Tang}{{\mathcal T}}
\newcommand{\Norm}{{\mathcal N}}
\newcommand{\Order}{{\mathcal O}}

\newcommand{\pd}{\partial}
\newcommand{\wo}{\setminus}
\newcommand{\dotarg}{\, \cdot \,}
\newcommand{\vecarg}[1]{\mathbf{#1}}

\DeclareMathOperator{\diag}{diag}
\DeclareMathOperator{\diver}{div}
\DeclareMathOperator{\trace}{trace}
\DeclareMathOperator{\cp}{cp}

\DeclareMathOperator{\codim}{codim}

\def\mypi{180}
\pgfmathsetmacro{\myphi}{\mypi/3}
\pgfmathsetmacro{\cosmyphi}{cos(\myphi)}
\pgfmathsetmacro{\sinmyphi}{sin(\myphi)}

\pgfmathsetmacro{\mypsi}{\mypi/8}
\pgfmathsetmacro{\cosmypsi}{cos(\mypsi)}
\pgfmathsetmacro{\sinmypsi}{sin(\mypsi)}

\journalname{}

\title{An Embedding Technique for the Solution of Reaction-Diffusion Equations on Algebraic Surfaces with Isolated Singularities}
\titlerunning{Reaction-Diffusion on Singular Curves}

\author{Parousia Rockstroh \and Thomas M\"{a}rz \and Steven J. Ruuth
	\thanks{Parousia Rockstroh was supported by a grant from NSERC Canada.
	\\ Thomas M\"{a}rz was supported by award KUK-C1-013-04 made by King Abdullah University of Science and Technology (KAUST).
	\\ Steven J. Ruuth was supported in part by a grant from NSERC Canada and by award KUK-C1-013-04 made by King Abdullah University of Science and Technology (KAUST).}}

\authorrunning{P. Rockstroh , T. M\"{a}rz , S. J. Ruuth} 

\institute{
	P. Rockstroh \at
	Cambridge Centre for Analysis, \\
	University of Cambridge, \\
	Cambridge CB3 0WA, UK. \\
	\email{\texttt{P.Rockstroh@maths.cam.ac.uk}}.
	\and
	T. M\"{a}rz \at
	Oxford Centre for Collaborative Applied Mathematics, \\
	University of Oxford, \\
	Oxford OX1 3LB, UK. \\
	\email{\texttt{maerz@maths.ox.ac.uk}}.
	\and
	S.J. Ruuth \at
	Mathematics Dept., \\
	Simon Fraser University, \\
	Burnaby, BC, Canada, V5A-1S6. \\
	\email{\texttt{sruuth@sfu.ca}}. }

\date{Manuscript as of \today}

\begin{document}
\maketitle

\begin{abstract}
In this paper we construct a parametrization-free embedding technique for numerically evolving reaction-diffusion PDEs defined on algebraic curves that possess an isolated singularity. 
In our approach, we first desingularize the curve by appealing to techniques from algebraic geometry. We create a family of smooth curves in higher dimensional space that correspond to
the original curve by projection. Following this, we pose the analogous reaction-diffusion PDE on each member of this family and show that the solutions (their projection onto the original domain) 
approximate the solution of the original problem. Finally, we compute these approximants numerically by applying the Closest Point Method which is an embedding technique for solving PDEs on smooth surfaces
of arbitrary dimension or codimension, and is thus suitable for our situation. In addition, we discuss the potential to generalize the techniques presented for higher-dimensional surfaces
with multiple singularities.
\end{abstract}

\section{Introduction}\label{sect:Intro}
In this paper we study the problem of evolving a reaction-diffusion PDE on algebraic surfaces that contain isolated singularities. 
Singular problems have been studied extensively from an analytic perspective within the differential geometry community, originating with the work of Cheeger in  \cite{cheeger1979spectral},
where a functional calculus was given for the Laplace operator on a cone. This in turn led to fundamental solutions of the heat equation and wave equation on a cone
as given in \cite{cheeger1983spectral} and \cite{cheeger1982diffraction}. More recently, researchers such as Jeffres and Loya in \cite{jeffres2003regularity} 
have derived asymptotic expansions for solutions of geometric PDEs on conical manifolds. 
Our goal is to construct a parametrization-free embedding-based numerical method that approximates geometric PDEs posed on surfaces with cusp singularities.
In addition to being an interesting problem in its own right, the construction of such a numerical method might be applied to the approximation of PDEs on smooth but highly curved surfaces
such as the simulation of heat flow on filaments of high curvature or the development of cortical maps in which regions of high curvature occur along the sulcal lines.

Our approach will be to use and extend existing embedding techniques. Within the past decade several numerical embedding methods have been created for solving PDEs that are posed
on smooth surfaces. These include level-set methods such as those presented in \cite{greer2006fourth} and \cite{greer2006improvement} as well as the Closest Point Method as presented in
\cite{Merriman} and \cite{RuuthMerriman} and further extended in \cite{Colin1} and \cite{marz2012calculus}. 
However, the existing embedding techniques are low-order accurate or inconsistent when applied to PDEs on surfaces with singularities.
This is a result of the fact that within the class of embedding techniques, a smooth representation of the surface is needed in
a narrow computational band surrounding the given curve or surface. 
When the curvature of the surface increases it becomes more difficult, and computationally expensive, to discretize in such
a way that the region of high curvature is accurately captured. Moreover, when the curvature becomes infinite, as is the case with a cusp, classical
numerical embedding methods may become inaccurate or fail due to the loss of smoothness in the surface representation.

Our approach to evolving PDEs on singular surfaces is to evolve a modified PDE on a regularized version of the surface.
The first step here is to resolve the singularity in the underlying domain. 
To this end we employ a standard procedure from algebraic geometry which is known as ``blowing-up'' \cite{Hartshorne,Harris,smith2000invitation}. 
We construct the blow-up map in such a way that it produces a one-parameter family of smooth surfaces that approximate the original singular surface. 
The regularization leads to the same reaction-diffusion PDE on each of the smooth surfaces within the one-parameter family. 
Solving these PDE problems yields accordingly a family of functions which converges to the solution of the original problem posed on the singular surface.

By this approach, the domain of the PDE will change with the parameter. This is an effect which is disadvantageous for the numerics.  
Thus, we transform the PDE problems so that they all have the same smooth surface as a domain. 
This produces a one-parameter family of variable coefficient PDEs with coefficients depending now on the parameter which was introduced by the desingularization procedure.
The resulting equations are now ready for the application of an embedding technique using standard uniform grid numerical methods on a smooth domain.

In the construction of our numerical technique, we must be mindful of the fact that the blow-up procedure generally embeds the regularized surface into a higher dimensional space.
It is therefore necessary to choose a numerical method that is effective for arbitrary co-dimensional embeddings. The Closest Point Method is one such embedding method, as shown 
in \cite{marz2012calculus} and \cite{RuuthMerriman}, and is our method of choice.

The paper is organized as follows. In Section~\ref{sect:Surf}, we review the definitions relating to smooth embedded surfaces and surface-intrinsic differential operators.
Section~\ref{sect:Playground} introduces our problem of interest: a simple reaction-diffusion equation posed on the closed planar curve given by $y^2=x^3-x^4$.
This curve has a single cusp singularity at the origin. In addition, we construct an analytical solution to the given problem which will be used later to assess our numerical results.
In Section~\ref{sect:Regularization} we describe the regularization of the problem in detail. In particular, we demonstrate the blow-up procedure, construct analytical solutions
to the regularized problems, prove the convergence to the solution of the original reaction-diffusion equation, and describe the transformation which prepares the problem
for an embedding technique. In Section \ref{sect:Numerics}, we review the basics of the Closest Point Method and apply it to the regularized problem.
Moreover, we give convergence studies that demonstrate the robustness of the method. Section~\ref{sect:Extensions} discusses the extension of the presented
method to general algebraic surfaces with finitely many singularities. 

\section{Algebraic Surfaces and Surface Intrinsic Differentials}\label{sect:Surf}
This section begins by collecting the definitions relating to smooth embedded surfaces and surface-intrinsic differential operators.

In this paper, we consider real algebraic manifolds, which is precisely the class of curves and surfaces $S \subset \R^n$ given implicitly
as the solution of a system of $m = \codim S$ polynomial equations over the field $\R$. We write this system as
\begin{equation}
	\vphi( \x ) = 0 \quad \Leftrightarrow \quad \x \in S
\end{equation}
with $\vphi: \R^n \to \R^m$, where each component $\varphi_i$, $i \in \{1,\ldots,m\}$, is a polynomial in $n$ variables $x_j$, $j \in \{1,\ldots,n\}$.

\begin{definition}{(Tangent space)}\label{def:AlgTan}
The straight line $ \x + t \vec{v}$, $t \in \R$ is tangent to the surface $S$ at $\x \in S$  if the root at $t=0$ of the polynomial $\vphi( \x + t \vec{v} )$ (a polynomial in $t$)
has a multiplicity of at least 2. By Taylor expansion
\begin{equation}
	\vphi( \x + t \vec{v} ) = t \; D\vphi(\x) \cdot \vec{v} + t^2 \; \vec{p}(t) \;,
\end{equation}
where $t^2 \; \vec{p}(t)$ is the polynomial remainder and $\vec{p}$ is some polynomial, we see that this multiplicity requirement is equivalent to
\begin{equation}\label{eqn:TangNorm}
	D\vphi(\x) \cdot \vec{v} = 0 \;.
\end{equation}
In this case we call $\vec{v}$ a tangent vector of $S$ at $\x$, and $\vec{v}$ is in the tangent space $\Tang_{\x} S$.
Moreover, by \eqref{eqn:TangNorm} the rows of $D\vphi(\x)$ span the normal space $\Norm_{\x} S$.
\end{definition}

Note that in the first part of Definition~\ref{def:AlgTan}, we define tangency without the use of limits or derivatives,
as is standard practice in algebraic geometry, see e.g. \cite{smith2000invitation}.
This will aid us in characterizing and resolving singularities as we will see later in the paper.

\begin{definition}{(Regularity/Singularity/Smoothness)}\label{def:Singularity}
	Let $S \subset \R^n$ be a real algebraic surface, then:
	\begin{enumerate}[(a)]
		\item a point $\x \in S$ is called \textbf{regular} if $D \vphi(\x)$, the Jacobian at $\x$, is a full rank matrix;
			in this case the columns of $D \vphi(\x)$ form a basis of the normal space $\Norm_{\x} S$,
		\item a point $\x \in S$ is called \textbf{singular} if $D \vphi(\x)$ is rank-deficient,
		\item the surface $S$ is called \textbf{smooth} if all points $\x \in S$ are regular.
	\end{enumerate}
\end{definition}

Given a regular point $\x \in S$ we set $\No(\x) := D\vphi(\x)^T$, hence $\No(\x)$ is the matrix that contains the normal vectors as given by the implicit description.
Moreover, we denote the orthogonal projector that projects onto the tangent space $\Tang_{\x} S$ of $\x \in S$ by $\Proj(\x)$.
For regular points $\x \in S$ one can write the matrix $\Proj(\x)$ as
\begin{equation}\label{eqn:Proj}
	\Proj:S \rightarrow \R^{n\times n} \;, \quad \Proj(\x) = I - \No(\x) \dotarg \No(\x)^{\dag} \;, \quad  \No(\x) = D\vphi(\x)^T\;,
\end{equation}
where $\No(\x)^{\dag}:= (\No(\x)^T \No(\x))^{-1} \No(\x)^T$ denotes the pseudo-inverse of $\No(\x)$.
In this paper we will be particularly interested in cases in which surfaces posses isolated singular points, hence $\No$ is full-rank and $\Proj$ is
given by \eqref{eqn:Proj} almost everywhere.

We now proceed to define a calculus for surface functions that does not depend upon parametrizations.
Such a calculus has been introduced earlier (see for example \cite{ambrosio1994level,deckelnick2010h,dziuk2008eulerian,gray1993mathematical,marz2012calculus,milnor1997topology});
here we will use much of the same notation and terminology as \cite{marz2012calculus}.
We point out here that throughout the paper surface functions, i.e., functions of the form $f:S \rightarrow \mathbb{R}^m$, are not required to be polynomials.

\begin{definition}\label{def:DiffOps}
	Let $S$ be a real algebraic surface and $\x \in S$ be a regular point. We consider the following $C^1$-smooth surface functions:
	a scalar function $u:S \rightarrow \mathbb{R}$, a vector-valued function $f:S \rightarrow \mathbb{R}^m$, and a vector field $g:S \rightarrow \mathbb{R}^n$
	with $C^1$-smooth local extensions into $\R^n$ called $u_E$, $f_E$, $g_E$ (see \cite{marz2012calculus}).
	We define the surface gradient $\nabla_S$, the surface Jacobian $D_S$, and the surface divergence $\diver_{S}$ at the regular point $\x$ by:
	\begin{align*}
		\nabla_{S} u(\x)^{T} &:= \nabla u_{E} (\x)^T \cdot \Proj(\x) \;, \\
		D_{S} f(\x) &:= D f_{E} (\x) \cdot \Proj(\x) \;,\\
		\diver_{S} g(\x) &:= \trace(D_{S}g(\x)) = \trace(D g_{E}(\x) \cdot \Proj(\x)) \;.
	\end{align*}
	Here $\nabla$ is the gradient and $D$ the Jacobian in the embedding space $\R^n$ applied to the extensions of the surface functions.
\end{definition}

By combining the above operators, we may form higher order differential operators on surfaces. One such operator is the Laplace-Beltrami operator, i.e.,
the Laplace operator for surfaces, which we define below.
\begin{definition}\label{def:Lap}
	Given a smooth surface $S$, the Laplace-Beltrami operator is defined by:
	\begin{equation}
		\laplace_{S}u(\x) := \diver_{S} (\nabla_{S} u)(\x),
	\end{equation}
	for $C^l$-smooth, $l \ge 2$, scalar surface functions $u$.
\end{definition}

\begin{remark}
	In some literature (e.g. \cite{Chavel1}, \cite{Chavel2}, \cite{Jost}, \cite{gallot1988inegalites}) the Laplace-Beltrami operator is defined by
	$\laplace_{S} u = - \diver_{S} (\nabla_{S} u)$ to give a positive operator.
	This paper defines the Laplace-Beltrami operator to be a negative operator, as is standard when studying diffusion processes on surfaces \cite{rubinstein1998partial}.
\end{remark}

\section{Reaction-Diffusion on a Cuspidal Curve}\label{sect:Playground}
The differential operators given in Section~\ref{sect:Surf} are well-defined on smooth regions of a surface.
We are, however, interested in examining the case where the surface $S$ posses an isolated singularity.

For now, we focus on a sample problem which consists of a simple reaction-diffusion equation posed on a closed curve
that has an isolated cusp singularity at the origin $(x,y)=(0,0)$, but is smooth everywhere else.
This section also constructs a parametric solution to the sample problem which we use later in the paper to validate our results.

\subsection{The Reference Problem}\label{sect:RefProblem}
The domain of our reference problem is the algebraic curve given by
\begin{equation}\label{eqn:CuspCurve}
	y^2=x^3-x^4 \;.
\end{equation}
Formally, we let $\varphi: \R^2 \to \R$ be given by $\varphi(x,y) = y^2 - x^3 + x^4$ so that 
\begin{equation}
	S = \{ (x,y) \in \R^2: \varphi(x,y) = 0 \} \;.
\end{equation}
For the remainder of the paper, unless stated otherwise, we will take this as our definition of $S$.
Figure \ref{fig:CuspidalCurve} shows a plot of $S$. Clearly, the origin is a singular point.
We verify this using Definition~\ref{def:Singularity}:
the Jacobian of $\varphi$ is given by $D \varphi(x,y) = (-3x^2 + 4x^3, 2y)$. Since this is a $1 \times 2$-matrix, the singular points are those
where the Jacobian vanishes, $D\varphi(x,y) = 0$. We find that the only singular point is the origin $(x,y)=(0,0)$.
Consequently, the set $S \setminus \{0\}$ is a smooth curve by Definition~\ref{def:Singularity}(c).
\begin{figure}[h!]
	\centering
	\begin{tikzpicture}[scale=4,thick]
		\draw[->] (-.25,0) -- (1.25,0) node[right] {$x$};
		\foreach \x/\xtext in {.25/\frac{1}{4},.5/\frac{1}{2},.75/\frac{3}{4},1} \draw (\x ,1pt) -- (\x ,-1pt) node[anchor=north] {\scriptsize $\xtext$};

		\draw[->] (0,-.5) -- (0,.5) node[above] {$y$};
		\foreach \y/\ytext in {-.25/-\frac{1}{4}, 0, .25/\frac{1}{4}} \draw (1pt, \y) -- (-1pt,\y) node[anchor=east] {\scriptsize $\ytext$};

		\draw[color=red,domain=-3.141:3.141,smooth,variable=\t] plot ({.5*(1+cos(\t r))},{.25*sin(\t r)*(1+cos(\t r))});
	\end{tikzpicture}
	\caption{The red curve is given by $y^2=x^3-x^4$ and possesses an isolated singularity at the origin.} \label{fig:CuspidalCurve}
\end{figure}
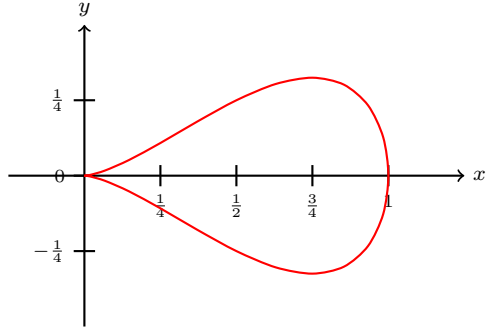

\begin{subequations}
	We now state the reference problem. Consider the reaction-diffusion problem
	\begin{align}\label{eqn:RefProbPDE1}
		& \pd_t u = \laplace_{S} u - \mu^2 u \;, \quad  t \in (0,\infty) \;,\; (x,y) \in S \setminus \{0\}
	\end{align}
	posed on the smooth part of $S$, i.e., on $S \setminus \{0\}$.
	By excluding the origin we introduce an interior boundary. The corresponding ``gap'' is closed by imposing the following boundary conditions
	\begin{align}
		& u(t,0_+) = u(t,0_-) \;, \label{eqn:RefProbBC1} \\
		& \pd_{\T} u(t,0_+) = \pd_{\T} u(t,0_-)  \;, \quad  t \in (0,\infty) \;. \label{eqn:RefProbBC2}
	\end{align}
	Condition \eqref{eqn:RefProbBC1} is a continuity requirement while \eqref{eqn:RefProbBC2} is Kirchhoff's circuit law regarding the fluxes into the interior boundary node.
	Here, $\T$ is the counter-clockwise tangent field of the curve $S \setminus \{0\}$ and $\pd_{\T} u(t,x,y) = \T^{T}(x,y) \cdot \nabla_{S}u(t,x,y)$ is the corresponding directional derivative.
	The positive and negative subscripts denote one-sided limits.
	The initial condition is given by:
	\begin{align}\label{eqn:RefProbIC}
		& u(0,x,y) = u_0(x,y)  \;,\quad (x,y) \in S \setminus \{0\} \;.
	\end{align}
	The initial boundary value problem (IBVP) specified by equations~\eqref{eqn:RefProbPDE1}--\eqref{eqn:RefProbIC} gives our reference problem.
\end{subequations}

\subsection{Parametric Solution of the Reference Problem} \label{sect:ArclengthSection}
We will now construct a parametric solution to the reference problem. This exact solution will be used to perform numerical convergence studies on our method.

Let $\bar{\gamma} : (0,L) \to S \setminus \{0\}$, $s \to \bar{\gamma}(s)$, denote a counter-clockwise arc-length parametrization of $S \setminus \{0\}$
with $\lim\limits_{s \to 0} \bar{\gamma}(s) = 0 = \lim\limits_{s \to L} \bar{\gamma}(s)$.
The counter-clockwise tangent field is parametrized by $\T \circ \bar{\gamma}(s) = \bar{\gamma}'(s)$, and one-sided limits of a surface function $f$
are given, in terms of $\bar{\gamma}$, by $f(0_-) = \lim\limits_{s \to 0} f \circ \bar{\gamma}(s) $, $f(0_+) = \lim\limits_{s \to L} f \circ \bar{\gamma}(s) $.

\begin{subequations}
	We denote by $\bar{u}(t,s) := u(t,\bar{\gamma}(s))$ the parametric solution of the reference problem corresponding to the arc-length parametrization $\bar{\gamma}$. Then, $\bar{u}$ must satisfy the PDE:
	\begin{align}\label{eqn:ArcPDE1}
		& \pd_t \bar{u} = \pd_s^2 \bar{u} - \mu^2 \bar{u} \;, \quad t \in (0,\infty) \;,\; s \in (0,L)
	\end{align}
	with periodic boundary conditions
	\begin{align}
		& \lim\limits_{s \to 0} \bar{u}(t,s) = \lim\limits_{s \to L} \bar{u}(t,s) \;, \\
		& \lim\limits_{s \to 0} \pd_s \bar{u}(t,s) = \lim\limits_{s \to L} \pd_s \bar{u}(t,s) \;, \quad t \in (0,\infty) \;, \label{eqn:ArcPDE4}
	\end{align}
	and initial condition
	\begin{align}
		& \bar{u}(0,s) = \bar{u}_0(s) = u_0(\bar{\gamma}(s)) \;,\quad s \in (0,L) \;.
	\end{align}
\end{subequations}
Hence, we obtain an arclength-parametric solution by Fourier series:
\begin{align}\label{eqn:ArclengthSolution}
	\bar{u}(t,s) &= e^{-\mu^2 t} \sum\limits_{m=-\infty}^{\infty} c_m \; e^{-\left(\frac{2\pi}{L}\right)^2 m^2 t} \; e^{i \left(\frac{2\pi}{L}\right) m \, s} \;, &
	c_m & = \frac{1}{L} \int\limits_{0}^{L}  \bar{u}_0 (s) \; e^{-i \left(\frac{2\pi}{L}\right) m \, s} \; ds \;.
\end{align}

An arclength parametrization may not be given, so we also express \eqref{eqn:ArclengthSolution} in terms of an arbitrary regular parametrization
$\gamma : (-\pi,\pi) \to S \setminus \{0\}$, $\theta \to \gamma(\theta)$. The arclength parameter $s = a(\theta)$ and the parameter $\theta$ are related via the arclength function:
\begin{equation}\label{eqn:AL1}
	a(\theta) = \int\limits_{-\pi}^{\theta} |\gamma'(\tau)| d\tau \;,
\end{equation}
so that $\bar{\gamma}(a(\theta)) = \gamma(\theta)$.
As a matter of notation, we let $\hat{u}(t,\theta)$ denote the parametric solution with respect to the parametrization $\gamma(\theta)$, i.e. $\hat{u}(t,\theta) = u(t,\gamma(\theta))$.
The relation between $\hat{u}$ and $\bar{u}$ is given by
\begin{equation*}
	\bar{u}(t,a(\theta)) = u(t,\bar{\gamma}(a(\theta))) = u(t,\gamma(\theta)) = \hat{u}(t,\theta) \;.
\end{equation*}
Consequently, by a change of variables we obtain the following series representation:
\begin{align}\label{eqn:SolOrigNonAL}
	\hat{u}(t,\theta) &= e^{-\mu^2 t} \sum\limits_{m=-\infty}^{\infty} c_m \; e^{-\left(\frac{2\pi}{L}\right)^2 m^2 t} \; e^{i \left(\frac{2\pi}{L}\right) m \, a(\theta)} \;, &
	c_m &= \frac{1}{L} \int\limits_{-\pi}^{\pi}  \hat{u}_0 (\theta) \; e^{-i \left(\frac{2\pi}{L}\right) m \, a(\theta)} \; a'(\theta)\; d\theta \;.
\end{align}
The parametrization that we use is $\gamma(\theta) = ( 1/2 (1 + \cos(\theta)) , 1/4 (1 + \cos(\theta)) \sin(\theta))$ (derived in Section~\ref{sect:SingularityResolution})
and is not based on arclength, so we will work with \eqref{eqn:SolOrigNonAL}.

\section{Regularization by Desingularization of the Domain}\label{sect:Regularization}

\subsection{Resolution of the Singularity in the Curve $y^2=x^3-x^4$}\label{sect:SingularityResolution}
In this section we resolve a singular point by a desingularization technique from algebraic geometry which is known as a blow-up (see for example, \cite{Hartshorne},\cite{Harris}, and \cite{smith2000invitation}).
For our curve $S$ (see Figure~\ref{fig:CuspidalCurve}) this will amount to viewing the cuspidal curve $S$ as a two-dimensional projection of a non-singular curve $\tilde{S}$ that lies in three-dimensional space.
The blow-up process also appears in the catastrophe theory literature, see for example \cite{arnold1981singularity}.

Recall, the cuspidal curve $S$ is given by the zero level set of $\varphi(x,y) := y^2 - x^3 + x^4$ which is singular at the origin. 
The Jacobian vanishes at the origin $D \varphi(0,0) = (0,0)$, hence the equality $D \varphi(0,0) \cdot v= 0$ holds for all $v \in \R^2$.
This means that the tangent space $\Tang_0 S$ is equal to all of $\R^2$ instead of being a one-dimensional subspace.

Now, we will work with the tangent space $\Tang_0 S$ in an algebraic fashion by examining all straight lines passing through the origin.
Specifically, we will consider
\begin{align}
	y^2 - x^3 + x^4 &= 0 \;,\label{eqn:Sys1_Cusp} \\
	y - \frac{z}{\varepsilon} x &= 0 \;, \label{eqn:Sys1_SLine}
\end{align}
where \eqref{eqn:Sys1_SLine} describes a family of straight lines through the origin with different slopes $z/\varepsilon \in \R$.
Figure~\ref{fig:BlowUp}~a) shows the situation for $\varepsilon=1$. 

Substituting \eqref{eqn:Sys1_SLine} into \eqref{eqn:Sys1_Cusp}, yields
\begin{align}
	x^2 \cdot \left( \left(\frac{z}{\varepsilon}\right)^2 - x + x^2 \right) &= 0 \label{eqn:Sys2_Cusp} \;, \\
	y - \frac{z}{\varepsilon} x &= 0 \;. \label{eqn:Sys2_SLine}
\end{align}
We may apply Definition~\ref{def:AlgTan} to \eqref{eqn:Sys2_Cusp} by noting that the substitution step is equivalent to evaluating $\varphi$ on the straight line parametrized by $(t,tz/\varepsilon)$
\begin{equation}
	\varphi\left(t,t \frac{z}{\varepsilon}\right) = t^2 \cdot \left( \left(\frac{z}{\varepsilon}\right)^2 - t + t^2 \right) \;.
\end{equation}
We see that $\varphi(t,tz/\varepsilon)$ as a polynomial in $t$ has a root at $t=0$ with multiplicity 2 and this is true for all possible choices of $z/\varepsilon$.
This implies that any straight line passing through the origin is tangent to the singular curve $S$ according to Definition~\ref{def:AlgTan}. 

\begin{figure}[h!]
\begin{minipage}{.525\textwidth}
	\centering
	\begin{tikzpicture}[scale=4,thick]
		\draw[->] (-.25,0) -- (1.25,0) node[right] {$x$};
		\foreach \x/\xtext in {.25/\frac{1}{4},.5/\frac{1}{2},.75/\frac{3}{4},1} \draw (\x ,1pt) -- (\x ,-1pt) node[anchor=north] {\scriptsize $\xtext$};

		\draw[->] (0,-.5) -- (0,.5) node[above] {$y$};
		\foreach \y/\ytext in {-.25/-\frac{1}{4}, 0, .25/\frac{1}{4}} \draw (1pt, \y) -- (-1pt,\y) node[anchor=east] {\scriptsize $\ytext$};

		\draw[color=blue,domain=-.25:1.25,smooth] plot (\x,{.5*\x}) node[anchor=south] {$z=\frac{1}{2}$};
		\draw[color=blue,domain=-.25:1.25,smooth] plot (\x,{.25*\x}) node[anchor=south] {$z=\frac{1}{4}$};
		\draw[color=blue,domain=-.25:1.25,smooth] plot (\x,{.125*\x}) node[anchor=south] {$z=\frac{1}{8}$};
		\draw[color=blue,domain=-.25:1.25,smooth] plot (\x,{-.5*\x}) node[anchor=south] {$z=-\frac{1}{2}$};
		\draw[color=blue,domain=-.25:1.25,smooth] plot (\x,{-.25*\x}) node[anchor=south] {$z=-\frac{1}{4}$};
		\draw[color=blue,domain=-.25:1.25,smooth] plot (\x,{-.125*\x}) node[anchor=south] {$z=-\frac{1}{8}$};

		\draw[color=red,domain=-3.141:3.141,smooth,variable=\t] plot ({.5*(1+cos(\t r))},{.25*sin(\t r)*(1+cos(\t r))});
	\end{tikzpicture}
	\vspace{51pt}

	\scriptsize{a) All straight lines through the origin are tangent to $S$ (red)}
\end{minipage}
\hfill
\begin{minipage}{.45\textwidth}
	\centering
	\includegraphics[width=\textwidth]{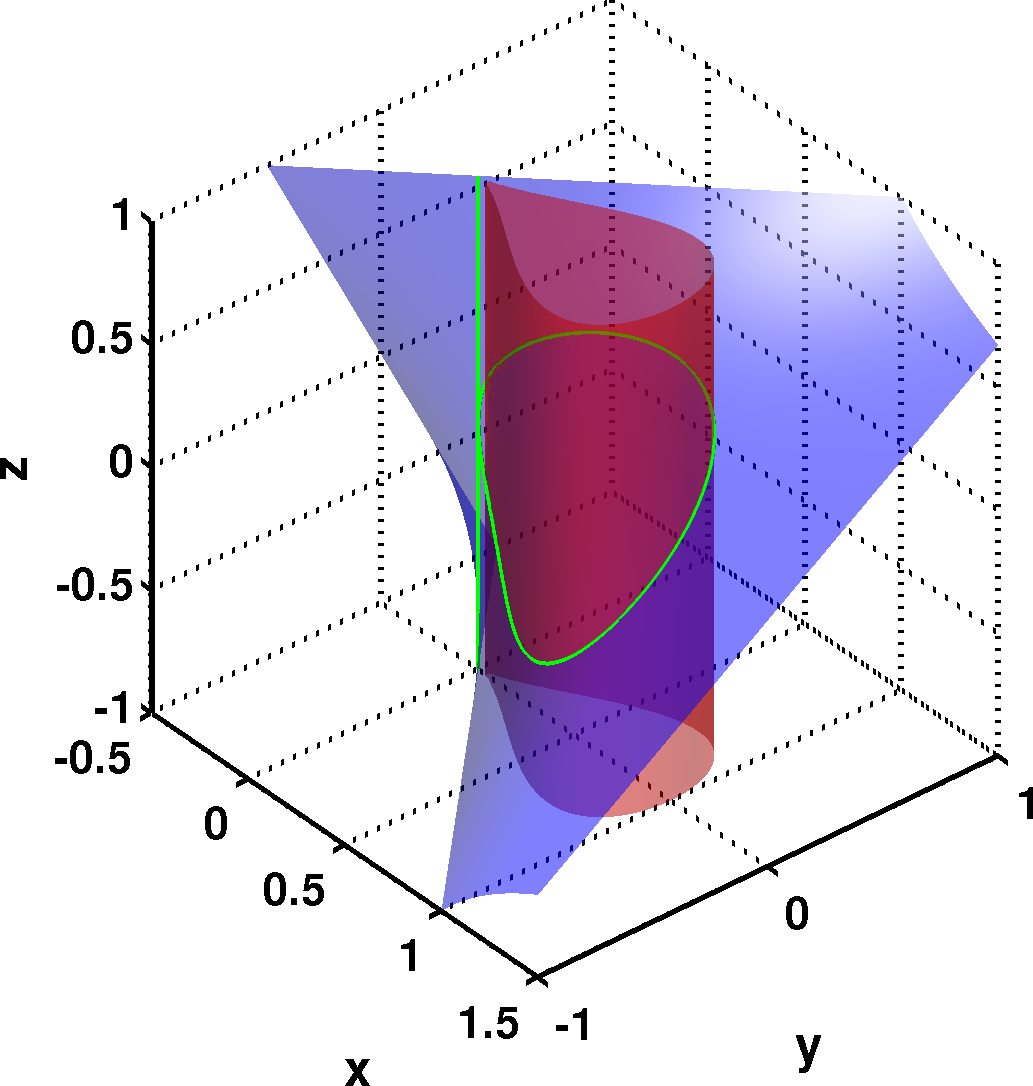}
	\vspace{1pt}

	\scriptsize{b) Intersection of the level-sets \eqref{eqn:Sys1_Cusp} and \eqref{eqn:Sys1_SLine} ($\varepsilon=1$)}
\end{minipage}
\medskip

\begin{minipage}{.525\textwidth}
	\centering
	\begin{tikzpicture}[scale=4.5,thick]
		\draw[color=red,domain=-3.141:3.141,smooth,variable=\t] plot ({\cosmyphi * .5 * (1+cos(\t r)) - \sinmyphi * .25 * sin(\t r) * (1+cos(\t r))},0,{\sinmyphi * .5 * (1+cos(\t r)) + \cosmyphi * .25 * sin(\t r) * (1+cos(\t r))});
		\draw[color=green,domain=-3.141:3.141,smooth,variable=\t] plot ({\cosmyphi * .5 * (1+cos(\t r)) - \sinmyphi * .25 * sin(\t r) * (1+cos(\t r))},{-.5*sin(\t r)},{\sinmyphi * .5 * (1+cos(\t r)) + \cosmyphi * .25 * sin(\t r) * (1+cos(\t r))});

		\draw[->] (-.25*\cosmyphi,0,-.25*\sinmyphi) -- (1.25*\cosmyphi,0,1.25*\sinmyphi) node[anchor=north] {$x$};
		\draw[->] (-.5*\sinmyphi,0,.5*\cosmyphi) -- (.5*\sinmyphi,0,-.5*\cosmyphi) node[anchor=west] {$y$};

		\draw[->] (0,-.5,0) -- (0,.5,0) node[anchor=south] {$z$};
	\end{tikzpicture}
	\vspace{47pt}

	\scriptsize{c) The original cuspidal curve (red) is the projection of the desingularized curve (green) onto the surface $z=0$}
\end{minipage}
\hfill
\begin{minipage}{.45\textwidth}
	\centering
	\includegraphics[width=\textwidth]{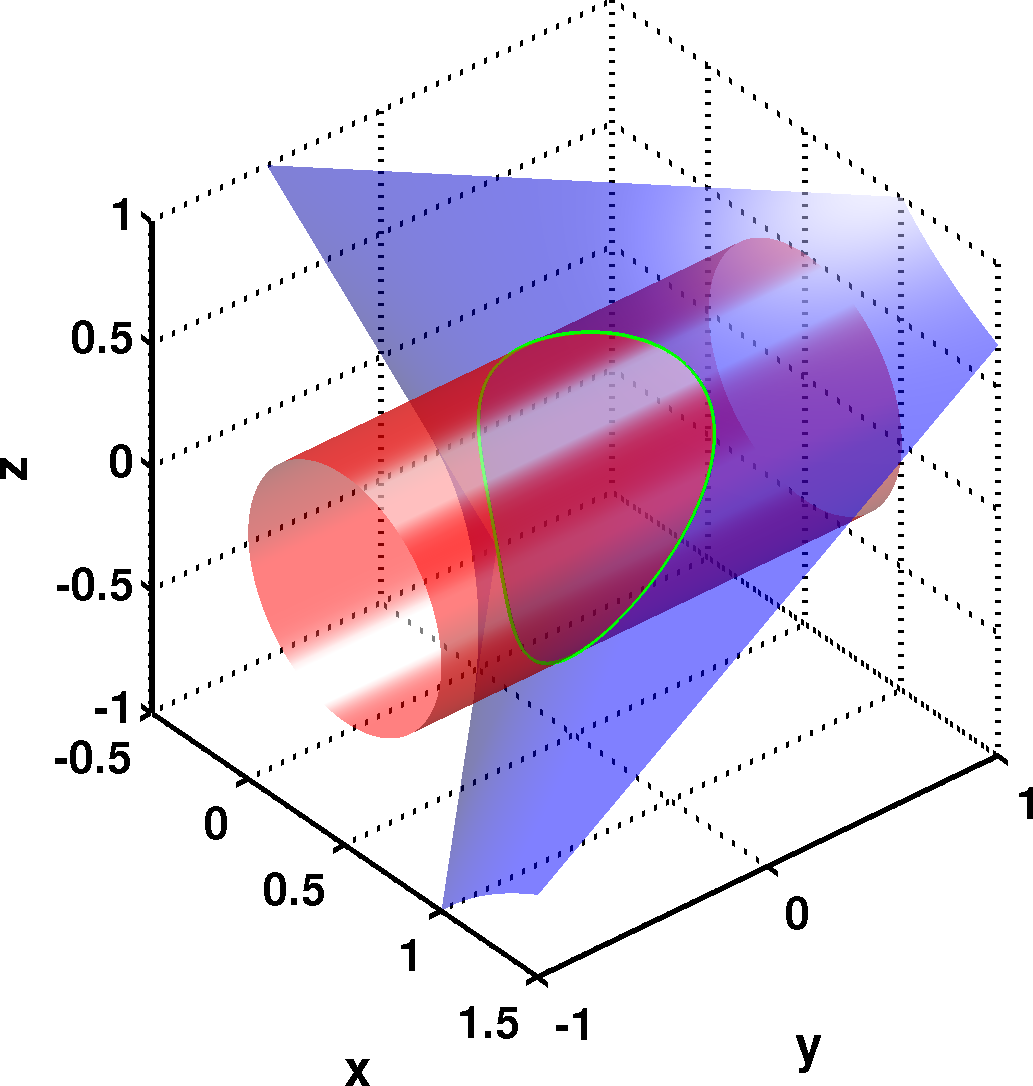}
	\vspace{1pt}

	\scriptsize{d) Intersection of the level-sets \eqref{eqn:Sys3_Desing} and \eqref{eqn:Sys3_SLine} ($\varepsilon=1$)}
\end{minipage}
\caption{Visualization of the blow-up process for a parameter value of $\varepsilon=1$.
a) shows the singular curve $S$ \eqref{eqn:Sys1_Cusp} in red and some tangents ($z$ is the slope) at the origin \eqref{eqn:Sys1_SLine} in blue.
b) visualizes the system of equations \eqref{eqn:Sys1_Cusp} (red surface) and \eqref{eqn:Sys1_SLine} (blue surface) where $z$ now represents the third dimension.
	The intersection of the two surfaces (green) contains the desingularized curve that we seek. Note that the intersection also contains the $z$-axis.
c) shows the desingularized curve in green, detached from any surfaces.
d) visualizes the system of equations \eqref{eqn:Sys3_Desing} (red surface) and \eqref{eqn:Sys3_SLine} (blue surface). In contrast to b), the red surface is now
	such that the intersection of the two surfaces (green) is exactly the desired desingularized curve.}\label{fig:BlowUp}
\end{figure}

The key to the blow-up is to interpret the system of equations \eqref{eqn:Sys2_Cusp} and \eqref{eqn:Sys2_SLine} in $\R^3$ with $z$ as the third variable.
Each of the equations \eqref{eqn:Sys2_Cusp} and \eqref{eqn:Sys2_SLine} defines a surface in $\R^3$ and the intersection of the surfaces
is a union of two curves. For a visual representation of this, see Figure~\ref{fig:BlowUp}~b). The first of these two curves is the $z$-axis, since
if $x=0$ then \eqref{eqn:Sys2_Cusp} is already satisfied while \eqref{eqn:Sys2_SLine} implies only $y=0$, hence all triples $(0,0,t)$ solve the system.
The second curve is obtained by considering the case where $x \neq 0$ and thus satisfies the system
\begin{align}
	\left(\frac{z}{\varepsilon}\right)^2 - x + x^2  &= 0 \;, \label{eqn:Sys3_Desing} \\
	y - \frac{z}{\varepsilon} x &= 0 \;. \label{eqn:Sys3_SLine}
\end{align}
since we can divide by $x$. Note that the origin $(0,0,0)$ is also a solution of \eqref{eqn:Sys3_Desing} and \eqref{eqn:Sys3_SLine} and thus belongs to both of these curves.
Comparing \eqref{eqn:Sys2_Cusp} and \eqref{eqn:Sys3_Desing} we see that the $x^2$-factor which caused the singularity according to our discussion above is now gone.
Hence, this second curve is the desingularization that we seek and we display it in Figure~\ref{fig:BlowUp}~c).   

Finally, we complete the square in \eqref{eqn:Sys3_Desing} with respect to $x$ and obtain an equivalent desingularized system:
\begin{align}
	\left(\frac{z}{\varepsilon}\right)^2 + \left(x - \frac{1}{2}\right)^2 - \frac{1}{4}  &= 0 \;, \label{eqn:Sys4_Desing}\\
	\varepsilon y - zx &= 0 \;. \label{eqn:Sys4_SLine}
\end{align}
Figure~\ref{fig:BlowUp}~d) illustrates the system of equations \eqref{eqn:Sys4_Desing} and \eqref{eqn:Sys4_SLine} as the intersection of two surfaces.
As can be seen from Figure~\ref{fig:BlowUp}~d) or the form of the equation, \eqref{eqn:Sys4_Desing} defines an elliptical cylinder.

The curve $\tilde{S}_{\varepsilon}$ described by the system of equations \eqref{eqn:Sys4_Desing} and \eqref{eqn:Sys4_SLine} can be parametrized
by $\tilde{\gamma}_{\varepsilon}: [-\pi,\pi) \to \R^3$, where
\begin{equation}\label{eqn:BlowUpCurve}
	\tilde{\gamma}_{\varepsilon}(\theta) := \Defo_{\varepsilon} \cdot \tilde{\gamma}(\theta) \quad \text{with} \quad
	\Defo_{\varepsilon} := \left[
							\begin{matrix}
								1 & 0 & 0 \\
								0 & 1 & 0 \\
								0 & 0 & \varepsilon
							\end{matrix}
							\right]
	\;,\quad
	\tilde{\gamma}(\theta) := \left(
										\begin{matrix}
											\frac{1}{2} (1 + \cos(\theta)) \\
											\frac{1}{4} (1 + \cos(\theta)) \sin(\theta) \\
											\frac{1}{2} \sin(\theta)
										\end{matrix}
										\right) \;.
\end{equation}
The system of equations \eqref{eqn:Sys4_Desing} and \eqref{eqn:Sys4_SLine} defines a one-parameter family of desingularized curves $\tilde{S}_{\varepsilon}$ depending on $\varepsilon$.
The parameter value $\varepsilon=1$ corresponds to the desingularized reference curve $\tilde{S}$ which is parametrized by $\tilde{\gamma}: [-\pi,\pi) \to \R^3$ of \eqref{eqn:BlowUpCurve}.
All further desingularizations $\tilde{S}_{\varepsilon}$ of our one-parameter family are deformations of $\tilde{S}$ where the linear deformation map is given in \eqref{eqn:BlowUpCurve} 
by the matrix $\Defo_{\varepsilon}$.

As $\varepsilon>0$ is reduced, the curve $\tilde{S}_{\varepsilon}$ will become an increasingly accurate approximation of the original singular curve $S$. Indeed,
$\tilde{S}_{0}$ is a copy of $S$ embedded in the $xy$-plane of $\R^3$.
In the situation where $\varepsilon=0$, the matrix $\Defo_{0}$ is the orthogonal projector onto the $xy$-plane, hence $\tilde{\gamma}_{0}$ parametrizes $\tilde{S}_{0}$ and
we obtain a parametrization $\gamma$ of the original singular curve $S \subset \R^2$ by dropping the trivial $z$-component of $\tilde{\gamma}_{0}$.
By reducing the domain of $\gamma$ to the open interval $(-\pi,\pi)$ we get the regular parametrization  of the smooth part $S \wo \{0\}$ 
mentioned in Section~\ref{sect:ArclengthSection}.

Finally, we note that the blow-up process provides us with a one-to-one correspondence between $(x,y,z) \in \tilde{S}_{\varepsilon}$ and $(x,y) \in S$.
The first map of this correspondence is clearly the projection
\begin{equation}\label{eqn:BlowDownMap}
	\tilde{S}_{\varepsilon} \to S \;, \qquad (x,y,z) \to (x,y) \;.
\end{equation}
The second map is the blow-up map, the inverse of this projection, and is given by
\begin{equation}\label{eqn:BlowUpMap}
	S \to \tilde{S}_{\varepsilon} \;, \qquad
	\left( \begin{matrix}
	      	x \\ y
	      \end{matrix} \right)
	\to
	\begin{cases}
		\left( x , y , \varepsilon \frac{y}{x} \right) &, \; x \neq 0 \\
		( 0 , 0 , 0 ) &, \; x=0 \;.
	\end{cases}
\end{equation}

\subsection{The Regularized Problem}\label{sect:RegProblem}
The regularization is carried out by considering the same reaction-diffusion PDE on the  new domain $\tilde{S}_{\varepsilon}$.
Because $\tilde{S}_{\varepsilon}$ is a smooth curve, we need not introduce an interior boundary. 
Instead, we have the following initial value problem (IVP) as the regularized problem
\begin{subequations}
	\begin{align}
		& \pd_t v_{\varepsilon} = \laplace_{\tilde{S}_{\varepsilon}} v_{\varepsilon} - \mu^2 v_{\varepsilon} \;, \quad  t \in (0,\infty) \;,\; (x,y,z) \in \tilde{S}_{\varepsilon} \;,\; \varepsilon > 0 \label{eqn:RegularizedPDE} \\
		& v_{\varepsilon}(0,x,y,z) = u_0(x,y)  \;,\quad (x,y,z) \in \tilde{S}_{\varepsilon} \;. \label{eqn:RegularizedIC}
	\end{align}
\end{subequations}
The initial condition \eqref{eqn:RegularizedIC} is constructed by lifting the initial condition of the original problem.
This is formally done via the one-to-one correspondence between $(x,y,z) \in \tilde{S}_{\varepsilon}$ and $(x,y) \in S$.
The projection map of \eqref{eqn:BlowDownMap} allows us to lift functions $f$ defined on $S$ to functions $g$ defined on $\tilde{S}_{\varepsilon}$ by $f \to g, \;  g(x,y,z) := f(x,y)$.

Because the solution $v_{\varepsilon}$ of the regularized problem does not have the same domain as the solution $u$ of the reference problem,
we introduce an approximant $u_{\varepsilon}$. The approximant $u_{\varepsilon}$ will be defined on the original domain $S$ and must tend to the desired solution $u$ as $\varepsilon \to 0$.
Note that the blow-map of \eqref{eqn:BlowUpMap} induces a pull-down of functions $g$ defined on $\tilde{S}_{\varepsilon}$ to functions $f$ defined on $S$ by $ g \to f, \; f(x,y) := g(x,y,\varepsilon y/x)$
(given that $x \neq 0$, otherwise we have $f(0,0)=g(0,0,0)$). 
We define the approximant $u_{\varepsilon}$ as the pull-down of $v_{\varepsilon}$, i.e.,
\begin{equation}\label{eqn:AppoxSol}
	u_{\varepsilon}: (0,\infty) \times S \wo \{0\} \to \R \;, \quad u_{\varepsilon}(t, x,y) := v_{\varepsilon} \left(t, x,y,\varepsilon \frac{y}{x} \right) \;,
\end{equation}
where $v_{\varepsilon}$ is the solution of \eqref{eqn:RegularizedPDE}--\eqref{eqn:RegularizedIC}.
In Section~\ref{sect:ConvTheo} we will prove that for any time $t >0$ the family $u_{\varepsilon}(t, \dotarg)$ converges uniformly to $u(t,\dotarg)$---the solution 
of the reference problem---as $\varepsilon$ tends to zero.

The parametric solution of the regularized problem is easily constructed.
The relevant parametrization $\tilde{\gamma}_{\varepsilon}$ of $\tilde{S}_{\varepsilon}$ is given by \eqref{eqn:BlowUpCurve}
and we denote the parametric solution by $\tilde{v}_{\varepsilon}(t,\theta) := v_{\varepsilon}(t,\tilde{\gamma}_{\varepsilon}(\theta))$.
Employing the same method as in Section~\ref{sect:ArclengthSection}, we obtain
\begin{align}\label{eqn:RegSolParam}
	\tilde{v}_{\varepsilon}(t,\theta) &= e^{-\mu^2 t} \sum\limits_{m=-\infty}^{\infty} c_{\varepsilon,m} \; e^{-\left(\frac{2\pi}{L_{\varepsilon}}\right)^2 m^2 t} \; e^{i \left(\frac{2\pi}{L_{\varepsilon}}\right) m \, a_{\varepsilon}(\theta)} \;, &
	c_{\varepsilon,m} &= \frac{1}{L_{\varepsilon}} \int\limits_{-\pi}^{\pi}  \hat{u}_0 (\theta) \; e^{-i \left(\frac{2\pi}{L_{\varepsilon}}\right) m \, a_{\varepsilon}(\theta)} \; a_{\varepsilon}'(\theta)\; d\theta
\end{align}
where $a_{\varepsilon}(\theta)$ is the arclength function corresponding to $\tilde{\gamma}_{\varepsilon}$, i.e.,
\begin{equation}\label{eqn:AL2}
	a_{\varepsilon}(\theta) := \int\limits_{-\pi}^{\theta} |\tilde{\gamma}_{\varepsilon}'(\tau)| d\tau \;.
\end{equation}
We will denote the length of the curve $\tilde{S}_{\varepsilon}$ by $L_{\varepsilon} := a_{\varepsilon}(\pi)$.
Finally, we note that setting the parameter value $\varepsilon$ to zero in \eqref{eqn:RegSolParam} gives $\tilde{v}_{0}(t,\theta) = \hat{u}(t,\theta)$, i.e. the parametric
solution of the reference problem \eqref{eqn:SolOrigNonAL}, even though the PDE \eqref{eqn:RegularizedPDE} is not defined for $\varepsilon=0$.

\subsection{Convergence Theory $(\varepsilon \to 0)$}\label{sect:ConvTheo}
We now turn our attention to proving the convergence of $u_{\varepsilon}(t,\dotarg)$ to $u(t,\dotarg)$.
As a first step, we establish continuity of $\tilde{v}_{\varepsilon}$ in $\varepsilon=0$ by the following lemma.

\begin{lemma} \label{lem:FourierConvergence}
	Let $\tilde{v}_{\varepsilon}(t,\dotarg)$ denote the parametric solution of the regularized problem as given in \eqref{eqn:RegSolParam} for parameter values $\varepsilon >0$.
	Let $\hat{u}(t,\dotarg)$ denote the parametric solution of the reference problem as given in \eqref{eqn:SolOrigNonAL}.
	Then for any fixed time $t > 0$, $\tilde{v}_{\varepsilon}(t,\dotarg)$ converges uniformly to $\hat{u}(t,\dotarg)$ as $ \varepsilon \to 0$.
\end{lemma}

\begin{proof}
	We begin by estimating the difference $a'_{\varepsilon}(\theta) - a'(\theta)$, where $a_{\varepsilon}(\theta)$ and $a_0(\theta)=a(\theta)$ are the arclength functions
	from \eqref{eqn:AL2} and \eqref{eqn:AL1} respectively.
	\begin{equation}\label{eqn:ALdiff}
		0 \leq a'_{\varepsilon}(\theta) - a'(\theta) = a'_{\varepsilon}(\theta) - a_0'(\theta) = |\tilde{\gamma}'_{\varepsilon}(\theta)| - |\tilde{\gamma}'_0(\theta)| \leq
		|\tilde{\gamma}'_{\varepsilon}(\theta) - \tilde{\gamma}'_0(\theta)| = \frac{|\varepsilon  \cos(\theta)|}{2} \leq  \frac{\varepsilon}{2} \;.
	\end{equation}
	This estimate implies directly 
	\begin{equation}\label{eqn:estimAL}
		\begin{aligned}
			& a'_{\varepsilon}(\theta) - a'(\theta) = \Order(\varepsilon) \; &&\text{uniformly in} \; \theta \;,\\
			& a_{\varepsilon}(\theta) - a(\theta) = \int\limits_{-\pi}^{\theta} a'_{\varepsilon}(\tau) - a'(\tau) \; d\tau = \Order(\varepsilon) \; &&\text{uniformly in} \; \theta \;, \\
			& L_{\varepsilon}-L = a_{\varepsilon}(\pi) - a(\pi) = \Order(\varepsilon) \;.
		\end{aligned}
	\end{equation}
	From \eqref{eqn:estimAL} we derive further implications which are dependent upon $m$ but are uniform in $\theta$:
	\begin{equation}\label{eqn:estimALrat}
		\begin{aligned}
			& \frac{2\pi m}{L_{\varepsilon}} = \frac{2\pi m}{L} (1+\Order(\varepsilon)) & \Rightarrow \qquad
			& 0 \geq \left(\frac{2\pi m}{L_{\varepsilon}}\right)^2 - \left(\frac{2\pi m}{L}\right)^2 =  \Order(m^2\varepsilon) \;, \\
			& \frac{2\pi m a_{\varepsilon}(\theta)}{L_{\varepsilon}} = \frac{2\pi m a(\theta)}{L} (1+\Order(\varepsilon)) & \Rightarrow \qquad
			& \frac{2\pi m a_{\varepsilon}(\theta)}{L_{\varepsilon}} - \frac{2\pi m a(\theta)}{L} = \Order(m \varepsilon) \;.
		\end{aligned}
	\end{equation}
	The quantities defined below will be used later on
	\begin{equation}
		\begin{aligned}
			\alpha_{\varepsilon,m} &:= \exp \left( -\left( \left(\frac{2\pi m}{L_{\varepsilon}}\right)^2 - \left(\frac{2\pi m}{L}\right)^2 \right) t \right) \;, &
			\delta_{\varepsilon,m}(\theta) &:= \exp \left( i \left( \frac{2\pi m a_{\varepsilon}(\theta)}{L_{\varepsilon}} - \frac{2\pi m a(\theta)}{L} \right) \right) \;,
		\end{aligned}
	\end{equation}
	but we estimate them now. Using the results of \eqref{eqn:estimALrat} we see that $\alpha_{\varepsilon,m} = 1+ \Order(m^2\varepsilon)$
	and that $\delta_{\varepsilon,m}(\theta) = 1+ \Order(m\varepsilon)$ for arbitrary but fixed $m \in \Z$ as $\varepsilon \to 0$.
	Hence, if $M \in \N$ denotes an arbitrary but fixed bound on the indices $m$, we can say that
	\begin{equation}\label{eqn:estimUniform}
		\begin{aligned}
			\alpha_{\varepsilon,m} &= 1+ \Order(\varepsilon) \;,\quad \text{for all} \quad |m| <= M \\
			\delta_{\varepsilon,m}(\theta) &= 1+ \Order(\varepsilon) \;,\quad \text{for all} \quad |m| <= M \;,\; \text{uniformly in} \; \theta.
		\end{aligned}
	\end{equation}
	Furthermore, since $\delta_{\varepsilon,m}$ lies on the complex unit circle, we have
	\begin{equation}\label{eqn:Unit}
		|\delta_{\varepsilon,m}(\theta) - 1| \leq 2 \;,\quad \text{for all} \quad m \in \Z \;,\; \text{uniformly in} \; \theta.
	\end{equation}
	We now proceed by examining the difference in the coefficients $c_{\varepsilon,m}$ given by \eqref{eqn:RegSolParam} and $c_m$ given by \eqref{eqn:SolOrigNonAL}:
	\begin{equation}
		c_{\varepsilon,m} - c_m =  \int\limits_{-\pi}^{\pi}  \hat{u}_0 (\theta) \;
		\left( e^{-i \left(\frac{2\pi}{L_{\varepsilon}}\right) m \, a_{\varepsilon}(\theta)} \; \frac{a_{\varepsilon}'(\theta)}{L_{\varepsilon}} -
			e^{-i \left(\frac{2\pi}{L}\right) m \, a(\theta)} \; \frac{a'(\theta)}{L}
		\right) \; d\theta \;.
	\end{equation}
	Estimating the modulus $|c_{\varepsilon,m} - c_m|$, we arrive at the following bound:
	\begin{equation}
		\begin{aligned}
			|c_{\varepsilon,m} - c_m| & \leq \int\limits_{-\pi}^{\pi}  |\hat{u}_0 (\theta)| \left| \left( \frac{a'(\theta)}{L} + \Order(\varepsilon)\right) - \delta_{\varepsilon,m}(\theta) \frac{a'(\theta)}{L} \right| \; d\theta \\ 
			& \leq \frac{\|\hat{u}_0\|_{\infty} \; \| a'\|_{\infty}}{L} \int\limits_{-\pi}^{\pi}  \left| 1  - \delta_{\varepsilon,m}(\theta) \right|  \; d\theta + \Order(\varepsilon) \;.
		\end{aligned}
	\end{equation}
	Using \eqref{eqn:Unit} we observe that there is a bound $B$ independent of $m$ such that $|c_{\varepsilon,m} - c_m| \leq B$. 
	Moreover, the same bound holds for $|c_m|$ and $|c_{\varepsilon,m}|$, specifically $|c_m| \leq B$ and $|c_{\varepsilon,m}| \leq B$.
	Using the estimate \eqref{eqn:estimUniform} with our bound on $m$, we also find that
	\begin{equation}
		|c_{\varepsilon,m} - c_m| = \Order(\varepsilon) \;,\quad \text{for all} \quad |m| <= M \;.
	\end{equation}
	Next, we estimate the difference $|\tilde{v}_{\varepsilon}(t,\theta)-\hat{u}(t,\theta)|$.
	Since $t > 0$ is fixed the factors $e^{-\left(\frac{2\pi}{L_{\varepsilon}}\right)^2 m^2 t}$ and $e^{-\left(\frac{2\pi}{L}\right)^2 m^2 t}$ decrease rapidly with increasing $|m|$.
	Hence for an arbitrary but fixed $\delta > 0$, we can choose $M \in \N$ large enough so that the tails of $|\tilde{v}_{\varepsilon}(t,\theta)|$ are small:
	\begin{equation}
		\begin{aligned}
			& \sum\limits_{m=M+1}^{\infty}  |c_{\varepsilon,m} \; e^{-\left(\frac{2\pi}{L_{\varepsilon}}\right)^2 m^2 t} \; e^{i \left(\frac{2\pi}{L_{\varepsilon}}\right) m \, a_{\varepsilon}(\theta)} |
			+ \sum\limits_{m=-\infty}^{-M-1}  |c_{\varepsilon,m} \; e^{-\left(\frac{2\pi}{L_{\varepsilon}}\right)^2 m^2 t} \; e^{i \left(\frac{2\pi}{L_{\varepsilon}}\right) m \, a_{\varepsilon}(\theta)} | \\
			& \leq \sum\limits_{m=M+1}^{\infty}  B \; e^{-\left(\frac{2\pi}{L_{\varepsilon}}\right)^2 m^2 t} + \sum\limits_{m=-\infty}^{-M-1} B \; e^{-\left(\frac{2\pi}{L_{\varepsilon}}\right)^2 m^2 t}
			= 2B \sum\limits_{m=M+1}^{\infty} e^{-\left(\frac{2\pi}{L_{\varepsilon}}\right)^2 m^2 t} \leq \frac{\delta}{3}.
		\end{aligned}
	\end{equation}
	A similar estimate holds for the tails of $\hat{u}(t,\theta)$.
	We note that the choice of $M$ depends also on $t$. Because $e^{-\left(\frac{2\pi}{L_{\varepsilon}}\right)^2 m^2 t}$ is monotonically decreasing in $m$ we can apply a standard integral test
	\begin{equation}
		\begin{aligned}
			\sum\limits_{m=M+1}^{\infty} e^{-\left(\frac{2\pi}{L_{\varepsilon}}\right)^2 m^2 t} & \leq \int\limits_M^{\infty} e^{-\left(\frac{2\pi}{L_{\varepsilon}}\right)^2 z^2 t} \; dz
			\leq \frac{L_{\varepsilon}}{2\pi \sqrt{2t}} \left( \sqrt{\frac{\pi}{2}} - \int\limits_0^{M} e^{-\frac{z^2}{2}} \; dz\right)
		\end{aligned}
	\end{equation}
	to obtain a suitable $M$. Now, using the estimates of the tails, we have:
	\begin{equation}
		\begin{aligned}
			|\tilde{v}_{\varepsilon}(t,\theta)-\hat{u}(t,\theta)| & \leq \sum\limits_{m=-M}^M
			\left| c_{\varepsilon,m} \; e^{-\left(\frac{2\pi}{L_{\varepsilon}}\right)^2 m^2 t} \; e^{i \left(\frac{2\pi}{L_{\varepsilon}}\right) m \, a_{\varepsilon}(\theta)}
			- c_{m} \; e^{-\left(\frac{2\pi}{L}\right)^2 m^2 t} \; e^{i \left(\frac{2\pi}{L}\right) m \, a(\theta)} \right| + \frac{2}{3} \delta \;.
		\end{aligned}
	\end{equation}
	Rearranging the remaining summands, we arrive at the following inequality:
	\begin{equation}
		\begin{aligned}
			|\tilde{v}_{\varepsilon}(t,\theta)-\hat{u}(t,\theta)| & \leq \sum\limits_{m=-M}^M
			\left\{ |c_{\varepsilon,m} - c_m| \alpha_{\varepsilon,m} + |c_m| | \alpha_{\varepsilon,m}-\delta_{\varepsilon,m}(\theta)| \right\} \; e^{-\left(\frac{2\pi}{L}\right)^2 m^2 t}  + \frac{2}{3} \delta \;.
		\end{aligned}
	\end{equation}
	Since the indices $m$ are between $-M$ and $M$, our bounds on $\alpha_{\varepsilon,m}$ and $\delta_{\varepsilon,m}(\theta)$ as of \eqref{eqn:estimUniform} apply.
	This gives us the following estimate which is uniform with respect to $\theta$
	\begin{equation}\label{eqn:estimFin}
		\begin{aligned}
			|\tilde{v}_{\varepsilon}(t,\theta)-\hat{u}(t,\theta)| & \leq
			\left\{ \Order(\varepsilon) (1+\Order(\varepsilon)) + B \Order(\varepsilon) \right\} \; (2M+1)  + \frac{2}{3} \delta \leq \Order(\varepsilon) + \frac{2}{3} \delta \;.
		\end{aligned}
	\end{equation}
	Finally, we choose $\varepsilon$ small enough such that the $\Order(\varepsilon)$-expression in \eqref{eqn:estimFin} is lower than $\delta/3$. This implies eventually
	\begin{equation}
		\|\tilde{v}_{\varepsilon}(t,\dotarg)-\hat{u}(t,\dotarg)\|_{\infty} \leq \delta \;,
	\end{equation}
	which proves the statement as $\delta > 0$ was arbitrary. \qed
\end{proof}

We now prove our main convergence result.

\begin{theorem}
	Let $u_{\varepsilon}(t, \dotarg)$ be given by \eqref{eqn:AppoxSol}, i.e., as the pull-down of the solution $v_{\varepsilon}(t, \dotarg)$
	of the regularized problem (equations~\eqref{eqn:RegularizedPDE}--\eqref{eqn:RegularizedIC}) at time $t$.
	Let $u(t,\dotarg)$ be the solution of the reference problem (equations~\eqref{eqn:RefProbPDE1}--\eqref{eqn:RefProbIC}) at time $t$.
	Then, for any fixed time $t>0$, $u_{\varepsilon}(t, \dotarg)$ converges uniformly to $u(t,\dotarg)$ as $\varepsilon \to 0$.
\end{theorem}
\begin{proof}
	We use the parametrizations $\gamma$ and $\tilde{\gamma}_{\varepsilon}$ from Section~\ref{sect:SingularityResolution}. 
	By construction, the first two components are equal $\tilde{\gamma}_{\varepsilon,i} = \gamma_i$, $i \in \{1,2\}$, while the third component $\tilde{\gamma}_{\varepsilon,3}$
	is equal to
	\begin{equation}
		\tilde{\gamma}_{\varepsilon,3} = \varepsilon \frac{\gamma_2}{\gamma_1} \;.
	\end{equation}
	If we parametrize $u_{\varepsilon}(t, \dotarg)$ by $\gamma$ then, using our convention $\hat{u}_{\varepsilon}(t,\theta) := u_{\varepsilon}(t, \gamma(\theta))$, we obtain
	\begin{equation}
		\hat{u}_{\varepsilon}(t, \dotarg) = u_{\varepsilon}(t, \gamma) = v_{\varepsilon}\left(t, \gamma_1, \gamma_2,  \varepsilon \frac{\gamma_2}{\gamma_1} \right) =
		v_{\varepsilon}\left(t, \tilde{\gamma}_{\varepsilon,1}, \tilde{\gamma}_{\varepsilon,2}, \tilde{\gamma}_{\varepsilon,3} \right) =
		v_{\varepsilon}\left(t, \tilde{\gamma}_{\varepsilon} \right) = \tilde{v}_{\varepsilon}(t, \dotarg) \;.
	\end{equation}
	Now, we appeal to Lemma~\ref{lem:FourierConvergence} to obtain
	\begin{equation}\label{eqn:LenCons}
		u_{\varepsilon}(t, \gamma) = \hat{u}_{\varepsilon}(t, \dotarg) \to \hat{u}(t, \dotarg) = u(t, \gamma) \quad \text{as} \quad \varepsilon \to 0 \;,
	\end{equation}
	where $\hat{u}(t, \dotarg) = u(t, \gamma)$ is the parametric solution of the reference problem.
	Since $\gamma$ is a regular parametrization of $S \wo \{0\}$ the statement follows from \eqref{eqn:LenCons}. \qed
\end{proof}

\subsection{Experimental Convergence Rate $(\varepsilon \to 0)$}\label{sect:ExpConv}
Thus far we have shown that $u_{\varepsilon}(t,\cdot)$ converges uniformly in space and pointwise in time to $u(t,\cdot)$ as $\varepsilon \rightarrow 0$.
For the purposes of developing a numerical method in later sections, we are interested in the rate at which $u_{\varepsilon}(t,\cdot)$ converges to $u(t,\cdot)$.
We investigate this by obtaining highly accurate approximations to $a(\theta)$ and $c_m$ (as defined in \eqref{eqn:SolOrigNonAL}) as well as $a_{\varepsilon}(\theta)$ and $c_{\varepsilon,m}$
(as defined in \eqref{eqn:RegSolParam}) by using Chebfun \cite{chebfunv4} in MATLAB. We then employ the first equation of \eqref{eqn:SolOrigNonAL} and the first equation of \eqref{eqn:RegSolParam}
where we restrict the summation to $-M \leq m \leq M$ with $M$ chosen sufficiently large such that the bound $e^{-\omega^2 M^2 t} \|\hat{u}_0\|_{\infty}$ is lower than machine accuracy.
We then compute the $l^\infty$-norm of the difference between $\hat{u}_{\varepsilon}(t,\cdot)$ and $\hat{u}(t,\cdot)$.
For our numerical experiments, we set $\mu = 1$ in \eqref{eqn:RegSolParam} and \eqref{eqn:ArclengthSolution}, and choose the initial conditions:
\begin{equation}\label{eqn:InitialCondition}
	\hat{u}_0(t,\theta) = \frac{\exp(4\cos^2(\theta))}{50}, \qquad \theta \in [-\pi,\pi).
\end{equation}

\begin{figure}[tbp]
	\begin{minipage}{.475\textwidth}
		\centering
		\includegraphics[scale=.4]{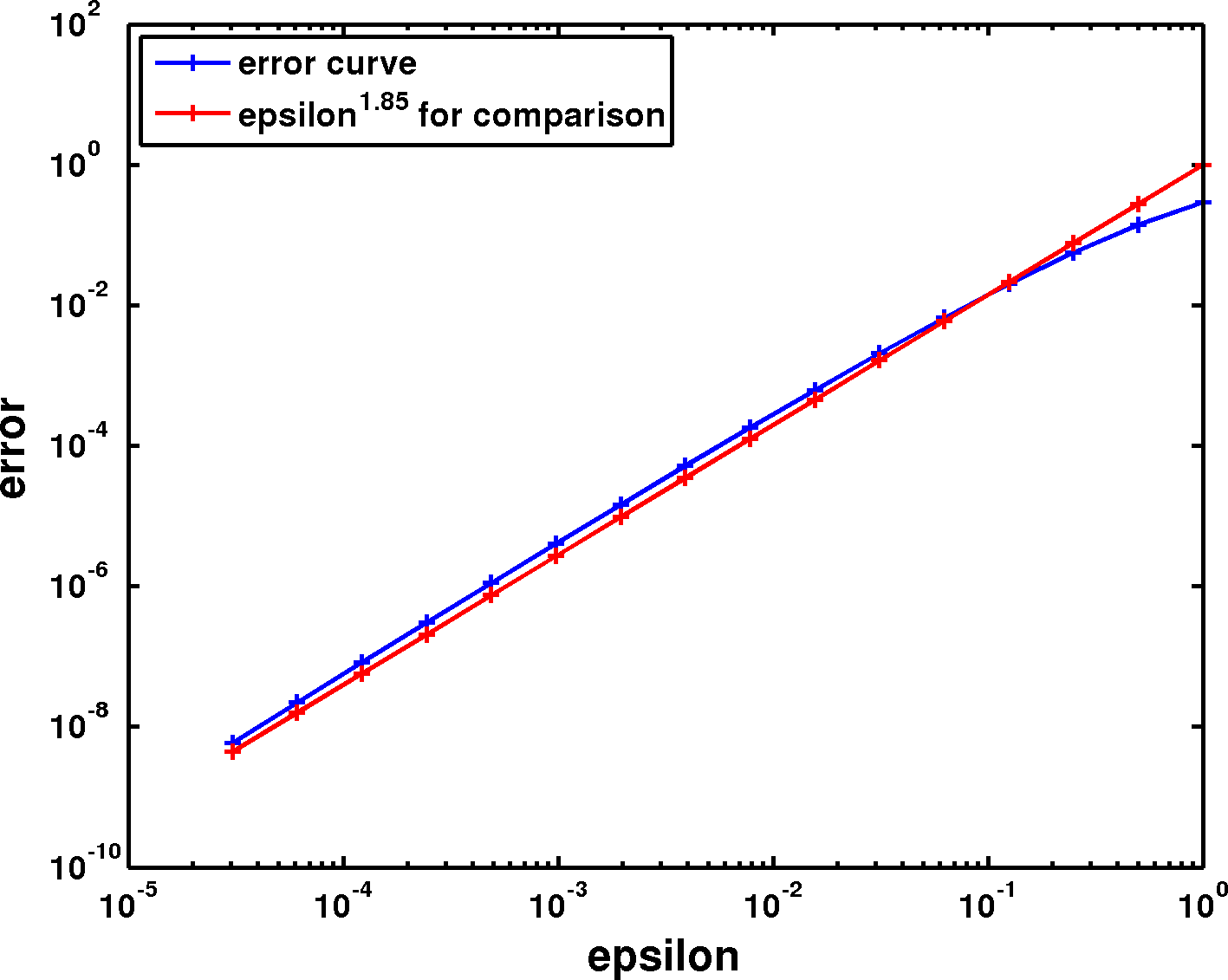}

		\scriptsize{a) Convergence of the solution of the heat equation for various choices of $\varepsilon$ for final time $t_f = .1$}
	\end{minipage}
	\hfill
	\begin{minipage}{.475\textwidth}
		\centering
		\includegraphics[scale=.4]{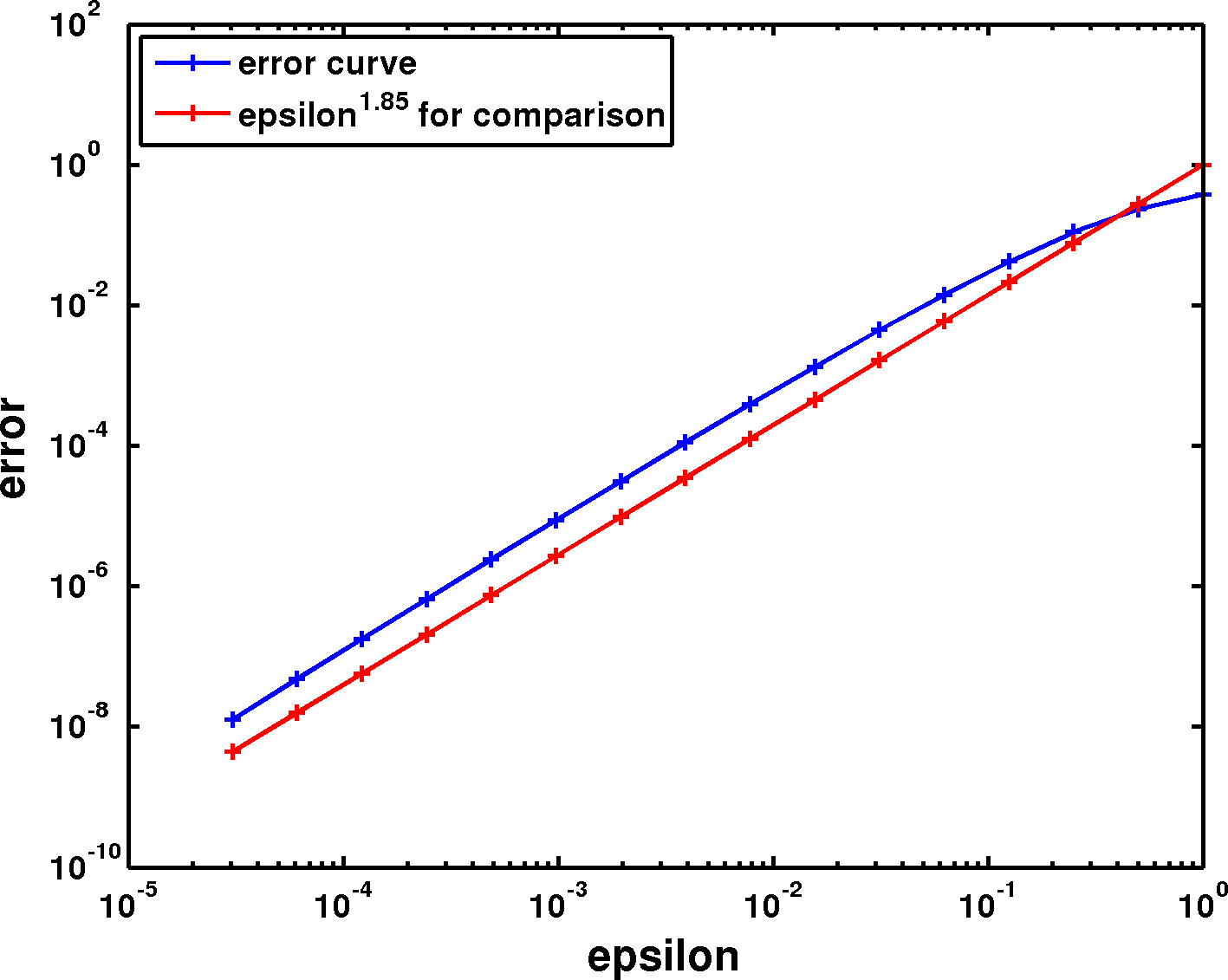}

		\scriptsize{b) Convergence of the solution of the heat equation for various choices of $\varepsilon$ for final time $t_f = .01$}
	\end{minipage}

	\begin{minipage}{.475\textwidth}
		\centering
		\includegraphics[scale=.4]{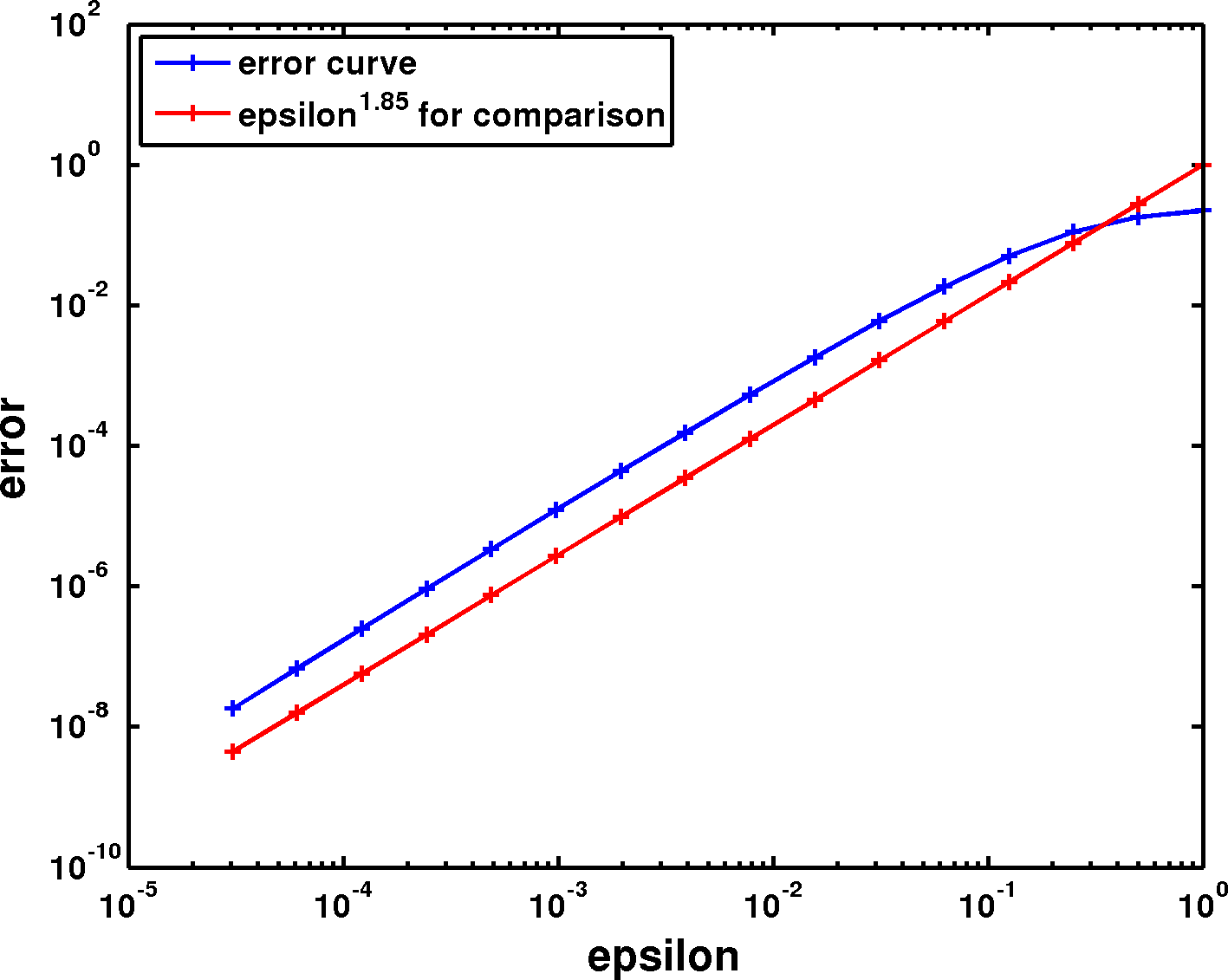}

		\scriptsize{c) Convergence of the solution of the heat equation for various choices of $\varepsilon$ for final time $t_f = .001$}
	\end{minipage}
	\hfill
	\begin{minipage}{.475\textwidth}
		\centering
		\includegraphics[scale=.4]{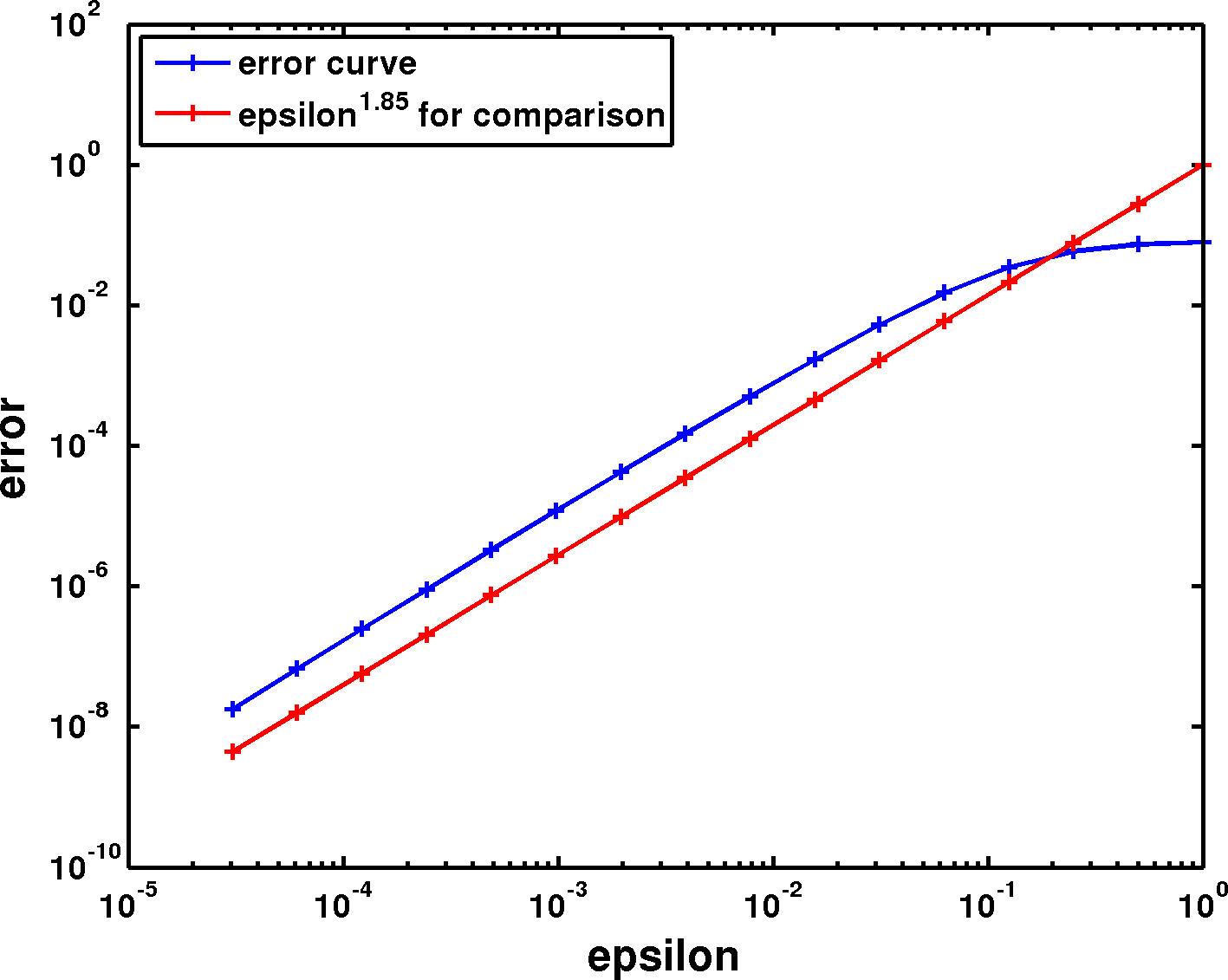}

		\scriptsize{d) Convergence of the solution of the heat equation for various choices of $\varepsilon$ for final time $t_f = .0001$}
	\end{minipage}
	\caption{Convergence results for various final times with various choices of $\varepsilon > 0$. The initial condition is given by \eqref{eqn:InitialCondition}.}
	\label{fig:FourierConv}
\end{figure}

We perform a regression analysis on the linear part of the data; see Figure~\ref{fig:FourierConv}.
Excluding the first six data points in the linear fit, we observe that the convergence rate of $\hat{u}_{\varepsilon}(t,\cdot)$ to $\hat{u}(t,\cdot)$ is $\varepsilon^{1.85}$ in each case.
However, we also note that in each case as the stop time $t_f$ decreases, the error curves in the double logarithmic plots of Figure~\ref{fig:FourierConv}
exhibit lower slopes for larger values of $\varepsilon$. 
In addition, we ran similar tests with other periodic initial conditions possessing the same order of smoothness.
In each case we found that the data exhibited similar convergence results with the convergence rate of $\hat{u}_{\varepsilon}(t,\cdot)$ to $\hat{u}(t,\cdot)$ ranging from $\varepsilon^{1.83}$ to $\varepsilon^{1.85}$.

\subsection{Transformation of the Regularized Problem}\label{sect:RegTrafo}
The regularized problem (equations \eqref{eqn:RegularizedPDE}--\eqref{eqn:RegularizedIC}) is a one-parameter family of problems.
In Section~\ref{sect:Numerics} these will be
solved for various choices of $\varepsilon > 0$ as $\varepsilon \to 0$ by an embedding technique.
However, as $\varepsilon \rightarrow 0$, the domain $\tilde{S}_{\varepsilon}$ approaches the singular domain $S$ which makes it
 difficult to solve the problem directly on the domains $\tilde{S}_{\varepsilon}$.
We resolve this issue by transplanting the problem onto a fixed curve $\tilde{S} = \tilde{S}_{1}$, in $\R^3$, 
and transforming the reaction-diffusion PDE to one involving variable coefficients.
To accomplish this task we introduce a new variable $w_{\varepsilon}$ via
the linear diffeomorphism induced by the matrix $\Defo_{\varepsilon}$ of \eqref{eqn:BlowUpCurve}:
\begin{equation}
	w_{\varepsilon}(t,\x) := v_{\varepsilon}(t, h_{\varepsilon}(\x) ) \;, \qquad h_{\varepsilon}:\tilde{S} \to \tilde{S}_{\varepsilon} \;, \; h_{\varepsilon}(\x) = \Defo_{\varepsilon} \cdot \x \;.
\end{equation}
The pull-down given in \eqref{eqn:AppoxSol} transforms accordingly to:
\begin{equation}
	u_{\varepsilon}(t,x,y) = w_{\varepsilon}\left(t,x,y,\frac{y}{x}\right) \;, \quad t \in (0,\infty) \;,\; (x,y) \in S \wo \{0\} \;.
\end{equation}
Now, Lemma~\ref{lem:DiffeoPDESwitch} below, tells us how to transform \eqref{eqn:RegularizedPDE} and \eqref{eqn:RegularizedIC} into a problem involving $w_{\varepsilon}$.

\begin{lemma} \label{lem:DiffeoPDESwitch}
	Let $\Gamma$ and $\tilde{\Gamma}$ be two smooth curves embedded in $\mathbb{R}^{n}$
	which are diffeomorphic to each other under the linear diffeomorphism $h:\tilde{\Gamma} \rightarrow \Gamma$, $h(\x)= \Defo \dotarg \x$.
	Let $u:\Gamma \rightarrow \R$ be a smooth function on $\Gamma$ and let $\tilde{u}:\tilde{\Gamma} \rightarrow \R$ denote its pull-back $\tilde{u} = u \circ h$ onto $\tilde{\Gamma}$.
	Then the pull-back of the Laplace-Beltrami operator $\laplace_{\Gamma} u \circ h$ onto $\tilde{\Gamma}$ transforms the following differential operator on $\tilde{\Gamma}$
	\begin{equation}\label{eqn:VariableGeometricPDE}
		\laplace_{\Gamma} u \circ h = \frac{1}{|\Defo \; \tilde{\T}|} \diver_{\tilde{\Gamma}} \left(\frac{1}{| \Defo \; \tilde{\T}|}\nabla_{\tilde{\Gamma}}\tilde{u}\right) \;,
	\end{equation}
	acting on $\tilde{u}$.
	In \eqref{eqn:VariableGeometricPDE} $\tilde{\T}$ is (one of the two possible choices of) the unit tangent vector field of $\tilde{\Gamma}$.
\end{lemma}

\begin{proof}
	According to Definitions~\ref{def:DiffOps} and \ref{def:Lap}, the Laplace-Beltrami operator on $\Gamma$ expands to
	\begin{equation}\label{eqn:Laplace1}
		\laplace_{\Gamma} u = \diver_{\Gamma}(\nabla_{\Gamma} u) = \trace( D \left(\Proj\nabla u \right) \Proj) \;,
	\end{equation}
	where $\Proj(\x)$ is the projector that projects onto the tangent space $\Tang_{\x} \Gamma$.
	The right-hand side of \eqref{eqn:Laplace1} involves a $C^1$-smooth local extension $u_E$ of $u$ (cf. Definition~\ref{def:DiffOps}),
	but we will simply use the identifier $u$ for both the function and its extension.
	Next, we expand the right-hand side of \eqref{eqn:Laplace1} by using the product rule:
	\begin{equation}\label{eqn:Laplace2}
		\laplace_{\Gamma} u = \trace \left(\Proj D^2 u \Proj + \left[ D(\Proj w) \right]|_{w=\nabla u} \cdot \Proj \right) \;.
	\end{equation}
	Since $\Gamma$ is a curve, the projector $\Proj(\x)$ is simply the outer product $\Proj(\x) = \T(\x) \cdot \T(\x)^T$. We use the product rule again
	to expand the Jacobian of the vector $\Proj w$ (where $w$ independent of $\x$)
	\begin{equation}
		D(\Proj w) = D (\T \T^T w) = D\T (\T^T w) + \T w^T D\T \;.
	\end{equation}
	With the last two results, the right-hand side of \eqref{eqn:Laplace2} becomes
	\begin{equation}
		\begin{aligned}
			\laplace_{\Gamma} u &= \trace \left(D^2 u \Proj \right) + \trace \left(D\T (\T^T \nabla u) \Proj \right) + \trace \left( \T(\nabla u^T D \T) \Proj \right) \\
			&= \T^T D^2 u \T +  (\T^T \nabla u) (\T^T D \T \T) + \nabla u^T D \T \T = \T^T D^2 u \T + \nabla u^T D \T \T \;. \label{eqn:Laplace3}
		\end{aligned}
	\end{equation}
	In the last equality we have used that $|\T|^2=1$ which implies $\T^T D\T = 0$.
	Now, we pull back:
	\begin{equation}\label{eqn:ResolvedLaplace}
		\laplace_{\Gamma} u \circ h = \T^{T}\circ h \; (D^2 u) \circ h \; \T \circ h + (\nabla u^{T}) \circ h \; D\T\circ h \; \T\circ h \;.
	\end{equation}
	The next step is to replace $(\nabla u) \circ h$ and $(D^2 u) \circ h$ with expressions involving derivatives of $\tilde{u} = u \circ h$.
	We obtain these expressions from the chain rule. Moreover, the linearity of $h$,
	\begin{equation}
		h(\x) = \Defo \; \x \;, \qquad Dh = \Defo \;, \qquad D^2 h = 0 \;,
	\end{equation}
	allows us to simplify in the following manner:
	\begin{align}
		\nabla \tilde{u}^{T} &= (\nabla u^{T}) \circ h \; Dh \qquad  &\Rightarrow \qquad \nabla u^{T} \circ h &= \nabla \tilde{u}^T \; \Defo^{-1}  \;, \label{eqn:GradU} \\
		D^2 \tilde{u} &= D h^{T} (D^2 u) \circ h \; Dh \qquad  &\Rightarrow \qquad D^2 u \circ h &= \Defo^{-T} \; D^2 \tilde{u} \; \Defo^{-1} . \label{eqn:D2U}
	\end{align}
	Finally, we must resolve $\T \circ h$. For our given $h$ a regular parametrization $\tilde{\gamma}$ of $\tilde{\Gamma}$ corresponds to a regular parametrization $\gamma$ of $\Gamma$
	via $\gamma = h(\tilde{\gamma})$. The linearity of $h$ and the fact that the unit tangent vector field has parametric form $\T \circ \gamma = \gamma'/|\gamma'|$ yield the following transformation rule:
	\begin{equation}\label{eqn:Ttrafo}
		\T \circ h(\tilde{\gamma})  = \T \circ \gamma = \frac{\gamma'}{|\gamma'|} = \frac{\Defo \; \tilde{\gamma}'}{|\Defo \; \tilde{\gamma}'|} =
		\frac{\Defo \; \tilde{\gamma}'/|\tilde{\gamma}'|}{|\Defo \; \tilde{\gamma}'|/|\tilde{\gamma}'|} = \frac{\Defo \; \tilde{\T} \circ \tilde{\gamma}}{|\Defo \; \tilde{\T} \circ \tilde{\gamma}|}
		\quad \Leftrightarrow \quad \T \circ h = \frac{\Defo \; \tilde{\T} }{|\Defo \; \tilde{\T} |} \;.
	\end{equation}
	Replacing the terms on the right-hand side of \eqref{eqn:ResolvedLaplace} with those of \eqref{eqn:GradU}, \eqref{eqn:D2U}, and \eqref{eqn:Ttrafo}, we arrive at
	\begin{equation}\label{eqn:ResolvedLaplace2}
		\laplace_{\Gamma} u \circ h = \frac{\tilde{\T}^T D^2 \tilde{u} \tilde{\T}}{|\Defo \; \tilde{\T}|^2} + \nabla \tilde{u}^{T} D \left(\frac{\tilde{\T}}{|\Defo \;\tilde{\T}|}\right) \frac{\tilde{\T}}{|\Defo \; \tilde{\T}|} \;.
	\end{equation}
	The right-hand side of \eqref{eqn:ResolvedLaplace2} involves only functions defined on $\tilde{\Gamma}$ and can be simplified further to:
	\begin{equation}
		\laplace_{\Gamma} u \circ h = \frac{1}{|\Defo \; \tilde{\T}|} \diver_{\tilde{\Gamma}} \left(\frac{1}{| \Defo \; \tilde{\T}|}\nabla_{\tilde{\Gamma}}\tilde{u}\right),
	\end{equation}
	thus completing the proof. \qed
\end{proof}

By Lemma~\ref{lem:DiffeoPDESwitch}, the regularized problem transforms to
\begin{subequations}
	\begin{align}
		& \pd_t w_{\varepsilon} = \beta_{\varepsilon} \diver_{\tilde{S}} \left( \beta_{\varepsilon} \nabla_{\tilde{S}} w_{\varepsilon} \right) - \mu^2 w_{\varepsilon} \;, \quad  t \in (0,\infty) \;,\; (x,y,z) \in \tilde{S}  \label{eqn:TPDE} \\
		& w_{\varepsilon}(0,x,y,z) = u_0(x,y)  \;,\quad (x,y,z) \in \tilde{S} \;, \label{eqn:TPDEIC}
	\end{align}
\end{subequations}
where the non-constant coefficient function $\beta_{\varepsilon}: \tilde{S} \to \R$ is given by
\begin{equation}
	\beta_{\varepsilon}(\x) = \frac{1}{|\Defo_{\varepsilon} \cdot \tilde{\T}(\x)|} \;,
\end{equation}
and $\tilde{\T}: \tilde{S} \to \R^3$ denotes the tangent field of $\tilde{S}$. Notice that the new domain $\tilde{S}$ is smooth and now independent of $\varepsilon$.

\section{An Embedding Technique for the Numerical Solution} \label{sect:Numerics}
In this section we solve the transformed regularized problem given by \eqref{eqn:TPDE} and \eqref{eqn:TPDEIC} without employing a parametrization;
instead, we embed the problem itself into the ambient space $\R^3$. The transformation of Section~\ref{sect:RegTrafo} made the domain $\tilde{S}$ independent of $\varepsilon$.
The advantage of this is that the choice of the computational grid is not affected by $\varepsilon$, but we have traded this advantage for a nearly 
singular coefficient since $\beta_{\varepsilon}(\vecarg{0}) = 1/\varepsilon$. This makes the new problem stiff.

\subsection{Embedding and Discretization of the Problem}
Our numerical calculations will be carried out using the Closest Point Method \cite{Merriman,RuuthMerriman,Colin1}, which is an embedding technique for solving PDEs posed on a given regular surface
$\Gamma$ of arbitrary codimension that is smoothly embedded in $\R^n$. We begin by reviewing some of the fundamental ideas behind the Closest Point Method.

\subsubsection{Basics of the Closest Point Method}
The Closest Point Method represents a given smooth surface using a closest point function \cite{marz2012calculus}.
A closest point function $\cp:B(\Gamma) \to \Gamma$ maps every point $\x$ of a band $B(\Gamma)$ (or tubular neighborhood) of $\Gamma$ to a point $\cp(\x) \in \Gamma$ and has the two properties:
\begin{enumerate}[a)]
	\item $\cp$ is a retraction, i.e., $\cp( \x ) = \x$ if $\x \in \Gamma$,
	\item $\cp$ is continuously differentiable with $D \cp( \x ) = \Proj(\x)$, i.e., the Jacobian of $\cp$ is equal to the projector $\Proj$ onto the tangent space of $\Gamma$ whenever the point $\x$ is on the surface $\Gamma$.
\end{enumerate}
The properties can be used to establish the following two principles \cite{marz2012calculus} which are fundamental to the Closest Point Method \cite{RuuthMerriman,Colin1}:
\begin{enumerate}[1.]
	\item Gradient principle: $\nabla_\Gamma u (\x) = \nabla[u \circ \cp](\x)$, if $\x \in \Gamma$,
	\item Divergence principle: $\diver_\Gamma g (\x) = \diver[g \circ \cp](\x)$, if $\x \in \Gamma$,
\end{enumerate}
where $u:\Gamma \to \R$ is a scalar surface function and $g:\Gamma \to \R^n$ is a surface vector field.
The Gradient and Divergence principles allow us to replace surface intrinsic differentials with Cartesian differentials 
applied to the closest point extensions $u \circ \cp$ and $g \circ \cp$.
Since closest point extensions are defined on $B(\Gamma)$ these principles when applied to a PDE posed on $\Gamma$ yield an embedding equation posed on the band $B(\Gamma)$.

We now elaborate further on our construction of the closest point function.

\subsubsection{A Closest Point Function For the Desingularized Domain}\label{sect:CPF}
For the domain $\tilde{S}$ of problem \eqref{eqn:TPDE}--\eqref{eqn:TPDEIC} we have an implicit description 
\begin{align}
	\varphi(x,y,z) &:= z^2 + \left(x - \frac{1}{2}\right)^2 - \frac{1}{4}  = 0 \;, \\
	\psi(x,y,z) &:= y - zx = 0
\end{align}
from Section~\ref{sect:SingularityResolution}.
Following the approach of \cite[Section 5]{marz2012calculus}, we construct---in a two-stage procedure---a closest point function $\cp: B(\tilde{S}) \rightarrow \tilde{S}$ directly using the given $\psi$ and $\phi$.
Let $S_{\varphi}$ and $S_{\psi}$ denote zero-level surfaces of $\varphi$ and $\psi$ respectively. The intersection of these is exactly $\tilde{S}$.
The first stage is to map a given point $\x \in B(\tilde{S})$ onto the surface $S_{\psi}$.
This is done by intersecting the trajectory $\xi$ of steepest decent, i.e.,
\begin{equation}\label{eq:Steep}
	\xi' = -\nabla \psi \circ \xi \;,\qquad \xi(0) = \x
\end{equation}
with the surface $S_{\psi}$. This defines a unique point $\bar{\x} \in S_{\psi}$.
The second stage is then to map the new point $\bar{\x} \in S_{\psi}$ onto $\tilde{S}$. Here, we intersect
the trajectory $\eta$ of
\begin{equation}\label{eqn:psiIVP}
	\eta' = -\nabla_{S_{\psi}} \varphi \circ \eta \;,\qquad \eta(0) = \bar{\x}
\end{equation}
with $\tilde{S}$. The second stage is well-defined because $\nabla_{S_{\psi}} \varphi$ is always tangent to $S_{\psi}$, hence the trajectory $\eta$ lies completely on the surface $S_{\psi}$
and evolves towards $\tilde{S}$. The second intersection is the point $\cp(\x) \in \tilde{S}$.
This construction defines a differentiable retraction. The proof that this defines a closest point function can be found in \cite{marz2012calculus}.

Now we discuss how to define the band $B(\tilde{S})$ which forms the domain of $\cp$.
Since $\nabla \psi$ does not vanish we can solve for $\xi$ for arbitrary initial points $\x$.
However, $\nabla \varphi$ vanishes on the straight line $\{ (1/2,y,0) : y \in \R \}$. The intersection of this line with $S_{\psi}$ is the point $(1/2,0,0)$.
Consequently, $\nabla_{S_{\psi}} \varphi$ will vanish only at $(1/2,0,0)$ and thus we cannot solve for $\eta$ by \eqref{eqn:psiIVP} if $\bar{\x} = (1/2,0,0)$.
By ODE \eqref{eq:Steep} we observe that the set of points that are mapped to $\bar{\x} = (1/2,0,0)$ in the first step is given by
the trajectory $\hat{\xi}(\tau) =  (\cosh(\tau)/2,\tau, - \sinh(\tau)/2)$, $\tau \in \R$. This implies that the maximal band is $B_{\max}(\tilde{S})= \R^3 \wo \hat{\xi}(\R)$.
In practice we define and use a smaller band that is contained in $B_{\max}(\tilde{S})$; for example, the following band is also admissible under the above criteria:
\begin{equation}\label{eq:band}
	B(\tilde{S}) := \left\{ (x,y,z) \in \R^3 : \sqrt{ 5 \, \varphi(x,y,z)^2 + \psi(x,y,z)^2 } < \frac{1}{2} \right\}.
\end{equation}
This choice of band is depicted in Figure~\ref{fig:band}.

\begin{figure}[t]
	\begin{center}
		\includegraphics[width=.5\textwidth]{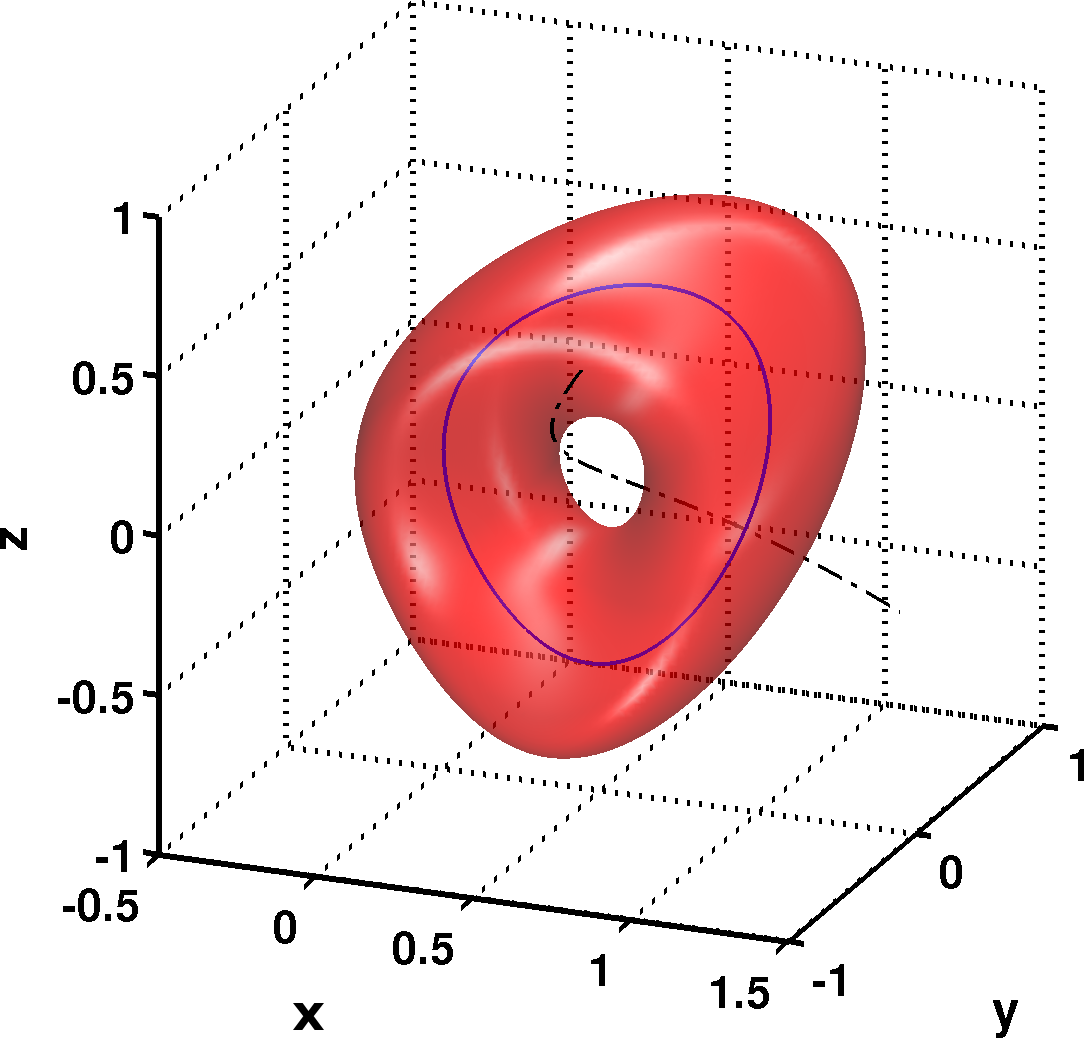}
	\end{center}
	\caption{A possible choice of the band $B(\tilde{S})$ (red) defined according to \eqref{eq:band}. The blue solid line is the curve $\tilde{S}$,
	while the black dashed line is $\hat{\xi}$, the set of points non-projectable by the closest point function $\cp$ as constructed in Section~\ref{sect:CPF}.
	The maximal possible band is $B_{\max}(S)= \R^3 \wo \hat{\xi}(\R)$.}\label{fig:band}
\end{figure}

\subsubsection{The Embedding Equation}
The embedding equation is obtained by applying the Divergence and Gradient principles to the original problem.
For this purpose we define the linear extension operator $E$ by $Eu(\x) := u \circ \cp(\x)$. The benefit of this is that the Divergence and Gradient principles
may be written as $E \diver_{\tilde{S}} g = E \diver E g$ and $E \nabla_{\tilde{S}} Eu$, see \cite{marz2012stab}.
We now apply the extension operator to \eqref{eqn:TPDE} and \eqref{eqn:TPDEIC} and arrive at the following embedded evolution equation:
\begin{equation} \label{eqn:EmbededEvo}
	\begin{aligned}
		& \pd_t E w_{\varepsilon} = E L E w_{\varepsilon} - \mu^2 E w_{\varepsilon} \;,\quad t \in (0,\infty) \;,\; \x=(x,y,z) \in B(\tilde{S}) \;,\\
		& E w_{\varepsilon} (0, \x) = u_0 ( \cp(\x)_1 , \cp(\x)_2 ) \;,\quad \x=(x,y,z) \in B(\tilde{S}) \;,
	\end{aligned}
\end{equation}
where the operator $L$ is defined by
\begin{equation} \label{eqn:LfDef}
	L f :=  \hat{\beta}_{\varepsilon}  \diver E \left( \hat{\beta}_{\varepsilon}  \nabla f \right) \;, \quad \hat{\beta}_{\varepsilon}(\x) := \beta_{\varepsilon} \circ \cp(\x).
\end{equation}
Finally, we write the embedded evolution equation \eqref{eqn:EmbededEvo} in the equivalent form
\begin{subequations}
	\begin{align}
		&  \check{w}_{\varepsilon} (0, \x) = u_0 ( \cp(\x)_1 , \cp(\x)_2 ) \label{eqn:Init} \\
		&  \pd_t \check{w}_{\varepsilon} = (E L - \mu^2 I) \check{w}_{\varepsilon} \;,\quad t \in (0,\infty) \;,\; \x=(x,y,z) \in B(\tilde{S}) \;, \label{eqn:Evo}\\
		&  \check{w}_{\varepsilon} = E \check{w}_{\varepsilon} \;. \label{eqn:Ext}
	\end{align}
\end{subequations}
Here, we have replaced $E w_{\varepsilon}$ with $\check{w}_{\varepsilon}$ and used the fact that $E$ is idempotent \cite{marz2012calculus,marz2012stab,vonglehn2012mol}.
The operator $I$ in \eqref{eqn:Evo} denotes the identity operator.

\subsubsection{Discretization of the Embedding Equation}
We discretize the embedded problem given by \eqref{eqn:Init}--\eqref{eqn:Ext} in a similar manner to \cite{RuuthMerriman,marz2012stab}, i.e. \eqref{eqn:Evo} is discretized in time with
either a forward Euler (Algorithm~\ref{algo:expl}) or a backward Euler (Algorithm~\ref{algo:impl}) time step (of step size $\tau$) followed by an extension in order to satisfy \eqref{eqn:Ext}.
For the discretization in space, we use a uniform Cartesian grid $G_h$ of mesh width $h$ on the box $[-0.5 , 1.5] \times [-1 , 1] \times [-1,1]$.
Computations are performed within the banded grid $G_h \cap B(\tilde{S})$ and we denote a point in the grid by $\x_h$.
We further use the subscript $h$ for spatially discrete functions (vectors) and operators (matrices). The superscript $n$ is used for temporally discrete functions and refers to time $n \tau$
where $\tau$ is our time step-size.
The spatially discrete operators in Algorithms~\ref{algo:expl} and \ref{algo:impl} are $I_h$, $E_h$, $L_h$.
Here, $I_h$ is simply the identity matrix while $E_h$ is an interpolation matrix. Because $\cp(\x_h)$ is in general not a grid point,
we approximate the function value of a closest point extension $Ew(\x_h) = w \circ \cp(\x_h)$.
This is done by performing a tri-cubic interpolation with the discrete function data $w_h$ on the $4 \times 4 \times 4$-point neighborhood surrounding the point $\cp(\x_h)$ \cite{RuuthMerriman,Colin1}.
Finally, we expand the operator $L$ from \eqref{eqn:LfDef} in terms of partial derivatives and obtain $L_h$ as a composition of matrices:
\begin{equation} \label{DiscreteLh}
	L_h :=  B_h \cdot \left( D_{x,h} E_h B_h D_{x,h} + D_{y,h} E_h B_h D_{y,h} + D_{z,h} E_h B_h D_{z,h}  \right),
\end{equation}
where $B_h$ is a diagonal matrix representing the factor $\hat{\beta}_{\varepsilon}$. The matrices $D_{x,h}$, $D_{y,h}$, $D_{z,h}$ approximate the partial derivatives
$\pd_x$, $\pd_y$, $\pd_z$ by central differences, and $E_h$ is again the extension matrix.

\begin{table}[tbp]
	\centering
	\begin{tabular}{lcr}
		\begin{minipage}{.4\textwidth}
			\centering
			\begin{algorithm}[H]
				\begin{algorithmic}
					\State $\check{w}^0_{h,\varepsilon} := u_0 ( \cp(\x_h)_1 , \cp(\x_h)_2 )$
					\For{$n=0,1,2,\ldots$}
						\State $w_h := (I_h + \tau (E_h L_h - \mu^2 I_h)) \cdot \check{w}^n_{h,\varepsilon}$
						\State $\check{w}^{n+1}_{h,\varepsilon} := E_h \cdot w_h$
					\EndFor
				\end{algorithmic}
				\caption{Explicit Iteration}\label{algo:expl}
			\end{algorithm}
		\end{minipage}
		&~~~~~~~~~~~~~~&
		\begin{minipage}{.4\textwidth}
			\centering
			\begin{algorithm}[H]
				\begin{algorithmic}
					\State $\check{w}^0_{h,\varepsilon} := u_0 ( \cp(\x_h)_1 , \cp(\x_h)_2 )$
					\For{$n=0,1,2,\ldots$}
						\State $w_h := (I_h - \tau (E_h L_h - \mu^2 I_h))^{-1} \cdot \check{w}^n_{h,\varepsilon}$
						\State $\check{w}^{n+1}_{h,\varepsilon} := E_h \cdot w_h$
					\EndFor
				\end{algorithmic}
				\caption{Implicit Iteration}\label{algo:impl}
			\end{algorithm}
		\end{minipage}
	\end{tabular}
	\caption{Explicit and implicit closest point iteration for problem \eqref{eqn:Init}--\eqref{eqn:Ext}.}
\end{table}

\subsubsection{Consistency and Stability}
In \cite{marz2012stab}, it is shown that the truncation error of both the explicit and the implicit scheme is of order $\Order( h^2  +  h^4/\tau + \tau)$.
Thus we choose the time step-size to be $\tau=\Order(h^2)$ to obtain good accuracy.
This gives us an overall truncation error of $\Order(h^2)$ and, assuming stability, an anticipated convergence rate of $\Order(h^2)$.
The following heuristics guide our choice of time step-size.
Notice that problem \eqref{eqn:TPDE} is stiff: the function is largest at the origin $\beta_{\varepsilon}(\vec{0}) = 1/\varepsilon$, $\varepsilon \ll 1$,
hence the spatial operator behaves like
\begin{equation}
	\beta_{\varepsilon} \diver_{\tilde{S}} \left( \beta_{\varepsilon} \nabla_{\tilde{S}} w_{\varepsilon} \right) - \mu^2 w_{\varepsilon}
	\approx \frac{1}{\varepsilon^2} \laplace_{\tilde{S}} w_{\varepsilon}  - \mu^2 w_{\varepsilon} \quad \text{as} \quad \x \to \vec{0} \;.
\end{equation}
Since $G_{\text{expl}} := E_h (I + \tau( L_h - \mu^2 I))$, the iteration matrix of the explicit iteration (Algorithm~\ref{algo:expl}), is obtained from the forward Euler time-stepping scheme,
we arrive at the following time-step restriction:
\begin{equation}\label{eqn:explTstep}
	\tau = C_{\text{expl}} \cdot \varepsilon^2 \cdot h^2 \;.
\end{equation}
As we are interested in small parameter values $\varepsilon \ll 1$, this implies a very large number of time steps to reach a given stop time.
In contrast to this, the implicit iteration with $G_{\text{impl}} := E_h (I - \tau( E_h L_h - \mu^2 I))^{-1}$, which is based on backward Euler, 
is unconditionally stable (regarding $\varepsilon$) and can therefore be based on accuracy concerns alone, i.e.,
\begin{equation}\label{eqn:implTstep}
	\tau = C_{\text{impl}} \cdot h^2 \;.
\end{equation}
In our experiments, the time step-size $\tau$ is chosen according to \eqref{eqn:explTstep} with $C_{\text{expl}} = 1/4$
in the explicit case and according to \eqref{eqn:implTstep} with $C_{\text{impl}} = 1$ in the implicit case.

\subsection{Numerical Examples}
\begin{figure}[tbp]
	\begin{minipage}{.475\textwidth}
		\centering
		\includegraphics[width=.9\textwidth]{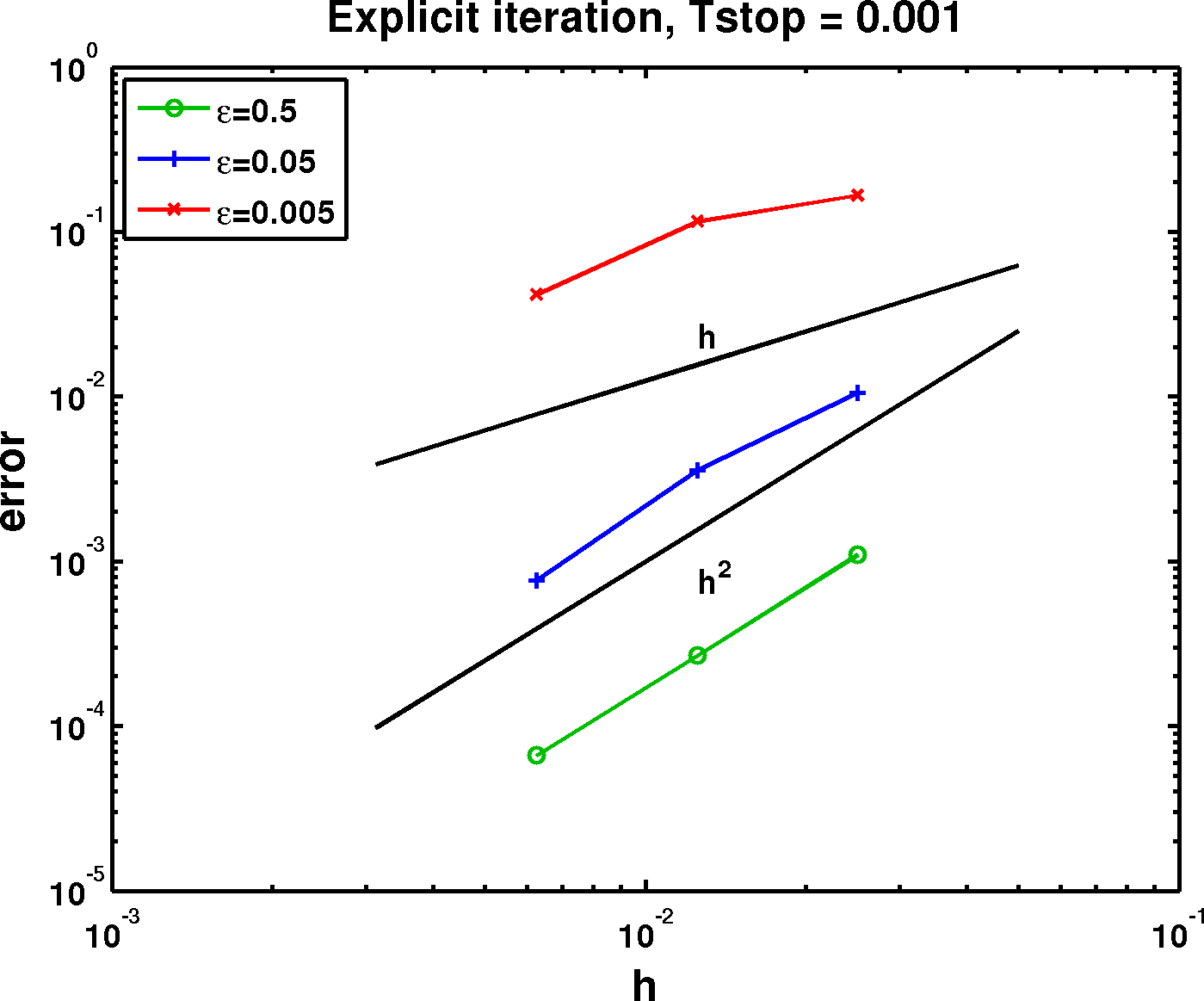}
	\end{minipage}
	\hfill
	\begin{minipage}{.475\textwidth}
		\centering
		\includegraphics[width=.9\textwidth]{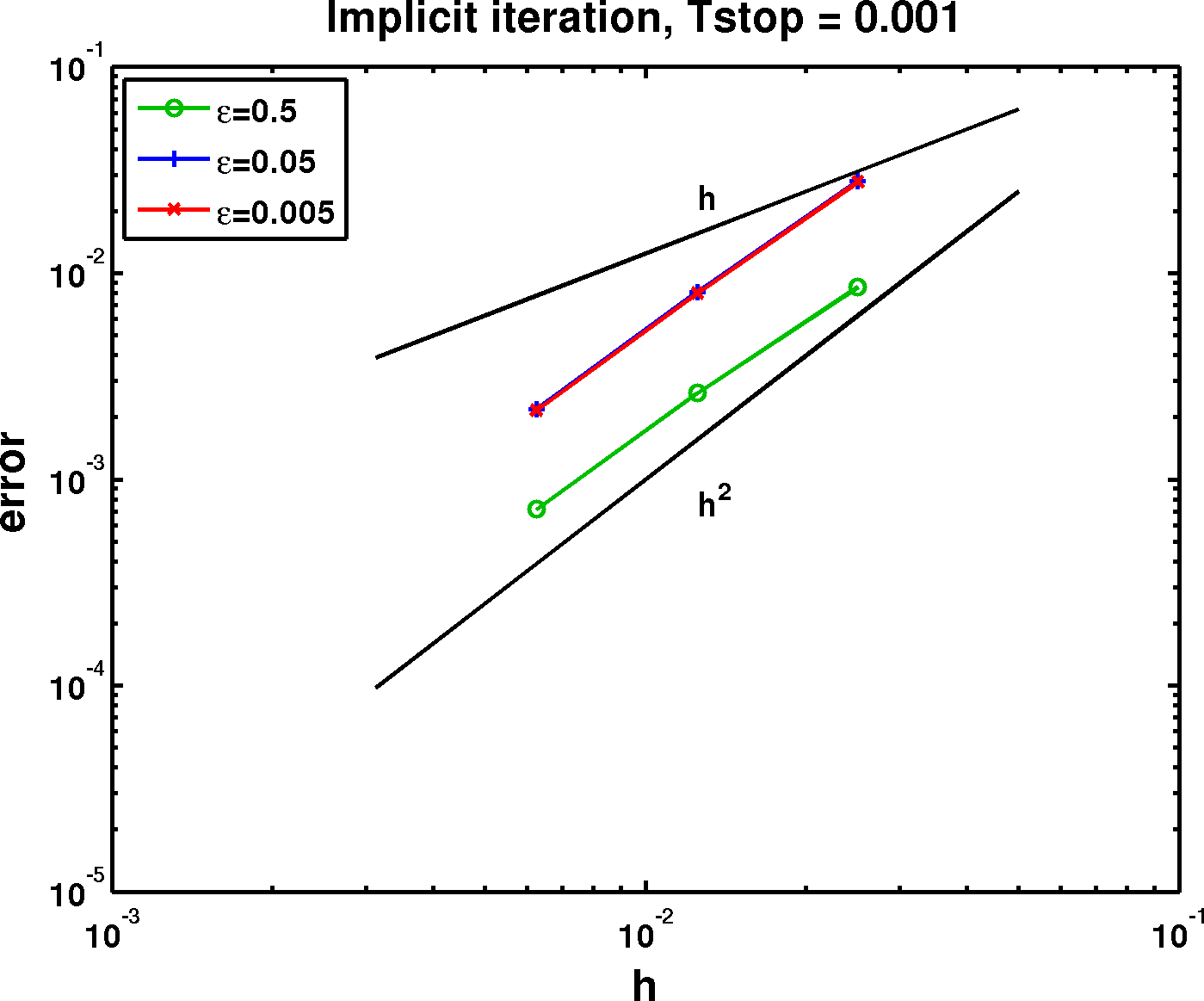}
	\end{minipage}
	\medskip

	\begin{minipage}{.475\textwidth}
		\centering
		\includegraphics[width=.9\textwidth]{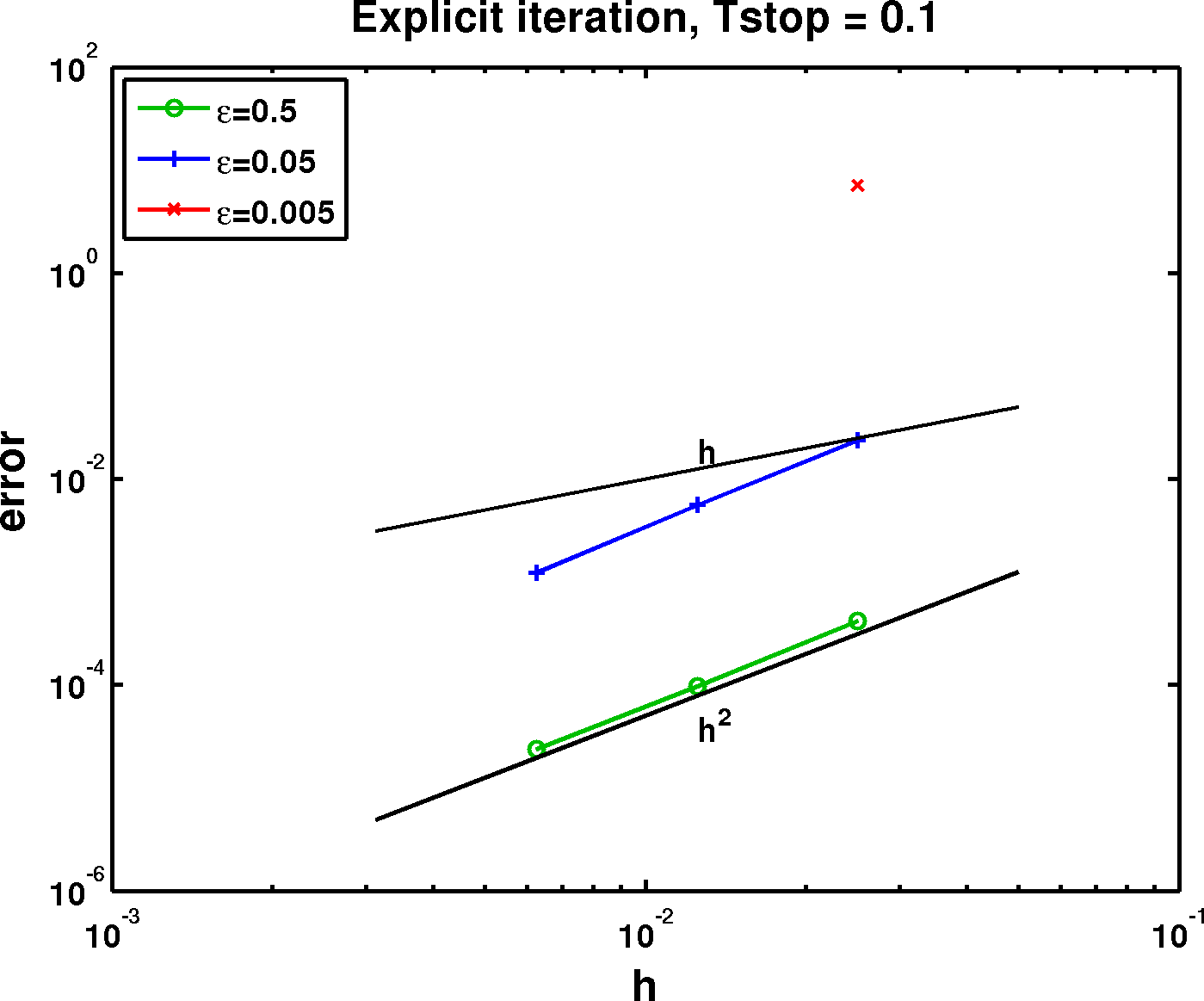}
	\end{minipage}
	\hfill
	\begin{minipage}{.475\textwidth}
		\centering
		\includegraphics[width=.9\textwidth]{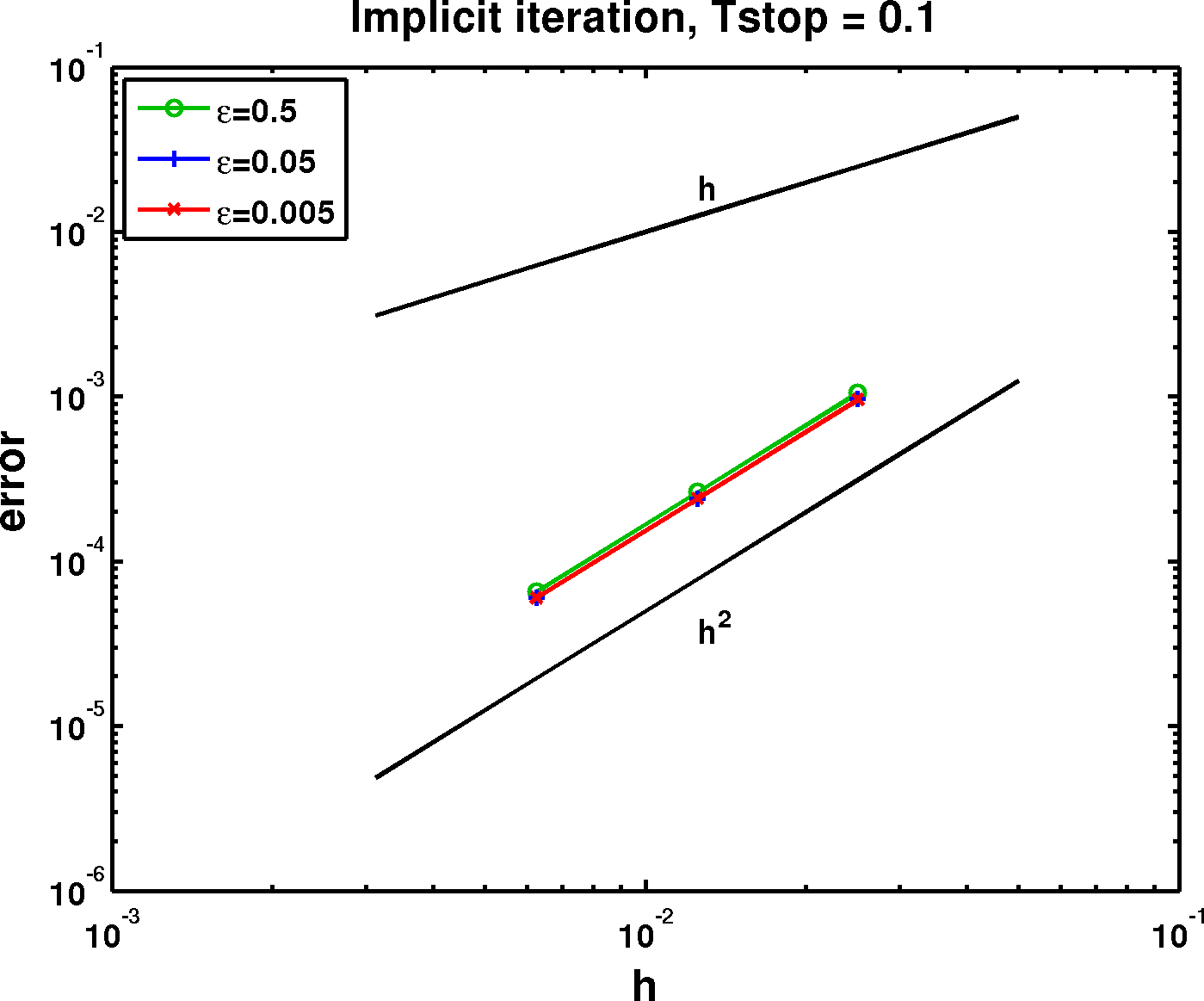}
	\end{minipage}
	\caption{Double logarithmic plot of the errors in Algorithm~\ref{algo:expl} (left column) and in Algorithm~\ref{algo:impl} (right column)
				measured in the $l^{\infty}$-norm at stop times $t=0.001$ (first row) and $t=0.1$ (second row). The initial condition is given in \eqref{eqn:ExpInitCond}
				and the reaction parameter is $\mu=1$.} \label{fig:CPconv}
\end{figure}
We run Algorithms~\ref{algo:expl} and \ref{algo:impl} for different choices of $\varepsilon \in \{0.5,0.05,0.005\}$ and measure errors by comparing the results to
the independently obtained parametric solutions of Section~\ref{sect:ExpConv}. The initial function is
\begin{equation}\label{eqn:ExpInitCond}
	u_0 (x,y) = \frac{\exp \left( 4 \cdot (2x-1)^2 \right)}{50} \;,
\end{equation}
which corresponds to \eqref{eqn:InitialCondition}. The parameter $\mu$ in the reaction term is set equal to $\mu=1$.

Figure \ref{fig:CPconv} shows the errors as a function of the mesh width $h$ in double logarithmic plots. The left column gives results for the explicit iteration while the right column gives results for
the implicit iteration. The first row gives the error at a stop-time of $t_f=0.001$, and the second row gives the error at a stop-time of $t_f=0.1$. In all cases,
the error was measured in the $l^{\infty}$-norm over the parameter domain $\theta \in [-\pi, \pi)$.
Examining the right column, we observe a convergence rate of $\Order(h^2)$, and that the error constant is relatively independent of the parameter
$\varepsilon$ which caused the stiffness. For the explicit iteration, we observe an $\Order(h^2)$-convergence for fixed $\varepsilon$, however the error constant grows as
$\varepsilon$ decreases. The number of time steps needed for the explicit iteration to reach a certain stop time becomes excessively large for small values of $\varepsilon$
(compare to \eqref{eqn:explTstep}). This explains why the red curve is missing in the bottom left plot of Figure~\ref{fig:CPconv}.
Clearly, the implicit iteration is much better suited for our stiff problem since the time step-size is independent of $\varepsilon$.

Summarizing, we see that the reference problem (equations equations~\eqref{eqn:RefProbPDE1}--\eqref{eqn:RefProbIC}) can be solved by combining the blow-up method with an embedding technique
that can handle arbitrary co-dimensions such as the Closest Point Method.

\section{Extensions and Generalizations}\label{sect:Extensions}

\subsection{Regularization by Desingularization of the Domain} \label{sect:BlowUpGenerality}
In Section~\ref{sect:Regularization} we regularized the reaction-diffusion problem posed on the cuspidal curve shown in Figure~\ref{fig:CuspidalCurve}.
The main step was the desingularization of the PDE domain via the resolution of the cusp singularity.
In general we consider reaction-diffusion PDEs posed on real algebraic surfaces with finitely many isolated singularities and want to apply the same regularization technique.

In order to resolve a single singularity of a real algebraic surface embedded in $\R^n$ we blow up $\R^n$ at the point $\vec{p} \in \R^n$ of singularity.
As in Section~\ref{sect:SingularityResolution} we investigate the tangent space $\Tang_p S$ of $S$ at $\vec{p}$ by adjoining equations that describe all straight lines through $\vec{p}$. 
Any straight line through $\vec{p}$ can be parametrized by $\vec{p} + t \vec{z}$ with a non-zero vector $\vec{z} \in \R^n$ and parameter $t \in \R$. 
If $\x$ is a point on this straight line then $\x-\vec{p}$ and $\vec{z}$ must be linearly dependent.
Consequently, the matrix 
\begin{equation}
\left[
	\begin{matrix}
		x_1-p_1  & z_1 \\
		\vdots & \vdots \\
		x_n-p_n  & z_n 
		\end{matrix}
\right]
\end{equation}
has rank $\le1$ and all its $2 \times 2$-minors vanish. 
The vanishing $2 \times 2$-minors yield an implicit description of the straight line through $\vec{p}$ in direction $\vec{z}$ by
\begin{equation}\label{eqn:BupSurf}
	(x_i-p_i) z_j - (x_j-p_j) z_i = 0 \;, \qquad 1 <= i < j <= n \;.
\end{equation}
In \eqref{eqn:BupSurf} there is plenty of redundancy since there are $n(n-1)/2$ equations in $\x$ to describe a straight line in $\R^n$.
Hence, we can choose a subset of $n-1$ linearly independent equations which we adjoin to our original system of polynomial equations that define the surface $S$. 

The so-called blow-up of $\R^n$ at the point $\vec{p} \in \R^n$ is given by viewing \eqref{eqn:BupSurf} as a system of equations in variables $\x$ and $\vec{z}$.
As on the one hand $\vec{z}$ is non-zero and on the other the length of $\vec{z}$ is irrelevant, it is standard in algebraic geometry to consider $\vec{z}$ as an 
element of the projective space $\PS^{n-1}$, i.e., $(\x,\vec{z}) \in \R^n \times \PS^{n-1}$. For more details on blow-ups the reader is referred to \cite{Hartshorne,Harris,smith2000invitation}.
Since we want to blow up in $\R^n \times \R^{n-1}$, we use one component of $\vec{z}$ as the parameter $\varepsilon$, i.e., $z = (\varepsilon, \hat{\vec{z}}) \in \R \times \R^{n-1}$,
and view \eqref{eqn:BupSurf} as a system of equations in variables $\x$ and $\hat{\vec{z}}$. This yields a one-parameter family of blow-ups in $\R^n \times \R^{n-1}$.

The final steps to desingularize are: substitute \eqref{eqn:BupSurf} into the original system of polynomial equations, pull out the factors that caused the singularity and
neglect them. In the end we obtain a system of polynomial equations which describes a family $\tilde{S}_{\varepsilon}$, $\varepsilon > 0$, of smooth surfaces that
correspond to the original surface $S$ by projection, i.e., $\varepsilon=0$. Given several singularities, we resolve one at a time by the same procedure. 
We will demonstrate this procedure later with more examples. In addition, we point out that blow-up calculations can be automatized with the computer algebra system SINGULAR \cite{DGPS}.  
  
This procedure for the resolution of singularities generalizes that of Section~\ref{sect:SingularityResolution}. Now, we would like to regularize, as in Section~\ref{sect:RegProblem},
by considering the same reaction-diffusion problem on $\tilde{S}_{\varepsilon}$. But this is not possible if the blow-up in $\R^n \times \R^{n-1}$ does not preserve the topology or connectivity
of the original domain of the PDE. We now give two examples of what can go wrong.
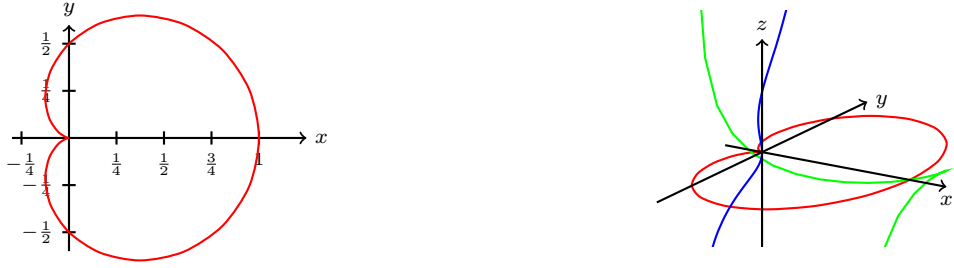
\begin{figure}[tbp]
	\begin{minipage}{.475\textwidth}
		\centering
		\begin{tikzpicture}[scale=2.5,thick]
			\draw[->] (-.3,0) -- (1.25,0) node[right] {$x$};
			\foreach \x/\xtext in {-.25/-\frac{1}{4} , .25/\frac{1}{4},.5/\frac{1}{2},.75/\frac{3}{4},1} \draw (\x ,1pt) -- (\x ,-1pt) node[anchor=north] {\scriptsize $\xtext$};

			\draw[->] (0,-.6) -- (0,.6) node[above] {$y$};
			\foreach \y/\ytext in {-0.5/-\frac{1}{2}, -.25/-\frac{1}{4}, .25/\frac{1}{4}, 0.5/\frac{1}{2}} \draw (1pt, \y) -- (-1pt,\y) node[anchor=east] {\scriptsize $\ytext$};
			\draw[color=red,domain=-3.141:3.141,smooth,variable=\t] plot ({cos(\t r)*(1+cos(\t r))*.5},{sin(\t r)*(1+cos(\t r))*.5});
		\end{tikzpicture}
	\end{minipage}
	\hfill
	\begin{minipage}{.475\textwidth}
		\centering
		\begin{tikzpicture}[scale=2.5,thick]
			\clip (-1,-.5) rectangle (1,.75);
			\draw[color=red,domain=-3.141:3.141,smooth,variable=\t] 
			plot({(\cosmypsi*cos(\t r)-\sinmypsi*sin(\t r))*(1+cos(\t r))*.5}, 
			0,
			{(\sinmypsi*cos(\t r) + \cosmypsi*sin(\t r))*(1+cos(\t r))*.5 });

			\draw[color=green,domain=-1.5:1.5,variable=\t] 
			plot({(\cosmypsi*cos(\t r)-\sinmypsi*sin(\t r))*(1+cos(\t r))*.5}, 
			{.25*sin(\t r)/cos(\t r)},
			{(\sinmypsi*cos(\t r) + \cosmypsi*sin(\t r))*(1+cos(\t r))*.5 });

			\draw[color=blue,domain=1.8:4.5,smooth,variable=\t]
			plot({(\cosmypsi*cos(\t r)-\sinmypsi*sin(\t r))*(1+cos(\t r))*.5}, 
			{.25*sin(\t r)/cos(\t r)},
			{(\sinmypsi*cos(\t r) + \cosmypsi*sin(\t r))*(1+cos(\t r))*.5 });

			\draw[->] (-.25*\cosmypsi,0,-.25*\sinmypsi) -- (1.25*\cosmypsi,0,1.25*\sinmypsi) node[anchor=north] {$x$};
			\draw[->] (-.75*\sinmypsi,0,.75*\cosmypsi) -- (.75*\sinmypsi,0,-.75*\cosmypsi) node[anchor=west] {$y$};

			\draw[->] (0,-.6,0) -- (0,.6,0) node[anchor=south] {$z$};
		\end{tikzpicture}
	\end{minipage}
	\caption{Left: cardioid in $\R^2$ given by \eqref{eqn:Cardioid}. Right: cardioid embedded in $xy$-plane of $\R^3$ (red) and its desingularization ($x > 0$ green, $x < 0$ blue) given by \eqref{eqn:CardReg}
				with $\varepsilon = 1/4$. The blow-up procedure tears the curve apart.}
	\label{fig:Cardioid}
\end{figure}

The first example is the cardioid (see left plot of Figure~\ref{fig:Cardioid}) which is given by
\begin{equation}\label{eqn:Cardioid}
	(x^2+y^2)^2 - x (x^2+y^2) - \frac{y^2}{4} = 0 \;.
\end{equation}
The procedure above or that of Section~\ref{sect:SingularityResolution} yields the following desingularization:
\begin{equation}\label{eqn:CardReg}
	\begin{aligned}
		\left( x \left(1+\frac{z^2}{\varepsilon^2}\right) - \frac{1}{2} \right)^2 - \frac{1}{4}\left(1+\frac{z^2}{\varepsilon^2}\right) &= 0\;,\\
		y - \frac{z}{\varepsilon} x &= 0 \;.
	\end{aligned}
\end{equation}
From this description we derive the following parametrization $\tilde{\gamma}_{\varepsilon}: [-\pi,\pi) \rightarrow \mathbb{R}^3$, 
$\tilde{\gamma}_{\varepsilon} = ((1+\cos(\theta))\cos(\theta)/2 , (1+\cos(\theta)) \sin(\theta)/2, \varepsilon \tan(\theta))$ of $\tilde{S}_{\varepsilon}$. 
Here, the blow-up in $\R^2 \times \R^1$ tears the curve apart, since $|\tan(\theta)|$ tends to infinity as $\theta$ approaches $\pm \pi/2$
(see right plot of Figure~\ref{fig:Cardioid}).
The blow-up changes the topology compared to the original curve which was a closed loop.
Consequently, we cannot use the same reaction-diffusion problem on $\tilde{S}_{\varepsilon}$ to regularize the problem.

\begin{figure}[tbp]
	\begin{minipage}{.475\textwidth}
		\centering
		\begin{tikzpicture}[scale=2.5,thick]
			\draw[->] (-1.25,0) -- (1.25,0) node[right] {$x$};
			\foreach \x/\xtext in {-1,-.75/-\frac{3}{4},-.5/-\frac{1}{2},-.25/-\frac{1}{4} , .25/\frac{1}{4},.5/\frac{1}{2},.75/\frac{3}{4},1} \draw (\x ,1pt) -- (\x ,-1pt) node[anchor=north] {\scriptsize $\xtext$};

			\draw[->] (0,-.6) -- (0,.6) node[above] {$y$};
			\foreach \y/\ytext in {-0.5/-\frac{1}{2}, -.25/-\frac{1}{4}, .25/\frac{1}{4}, 0.5/\frac{1}{2}} \draw (1pt, \y) -- (-1pt,\y) node[anchor=east] {\scriptsize $\ytext$};
			\draw[color=red,domain=-3.141:3.141,smooth,variable=\t] plot ({sin(\t r)},{cos(\t r)*sin(\t r)});
		\end{tikzpicture}
	\end{minipage}
	\hfill
	\begin{minipage}{.475\textwidth}
		\centering
		\begin{tikzpicture}[scale=2.5,thick]
			\draw[color=red,domain=-3.141:3.141,smooth,variable=\t] plot ({\cosmypsi * sin(\t r) - \sinmypsi * cos(\t r)*sin(\t r)},0,{\sinmypsi * sin(\t r) + \cosmypsi * cos(\t r)*sin(\t r)});
			\draw[color=green,domain=-3.141:3.141,smooth,variable=\t] plot ({\cosmypsi * sin(\t r) - \sinmypsi * cos(\t r)*sin(\t r)},{-.25*cos(\t r)},{\sinmypsi * sin(\t r) + \cosmypsi * cos(\t r)*sin(\t r)});

			\draw[->] (-1.25*\cosmypsi,0,-1.25*\sinmypsi) -- (1.25*\cosmypsi,0,1.25*\sinmypsi) node[anchor=north] {$x$};
			\draw[->] (-.5*\sinmypsi,0,.5*\cosmypsi) -- (.5*\sinmypsi,0,-.5*\cosmypsi) node[anchor=west] {$y$};

			\draw[->] (0,-.6,0) -- (0,.6,0) node[anchor=south] {$z$};
		\end{tikzpicture}
	\end{minipage}
	\caption{Left: figure-eight curve in $\R^2$ given by $y^2-x^2+x^4 = 0$. Right: figure-eight curve embedded in $xy$-plane of $\R^3$ (red) and its desingularization (green) given by \eqref{eqn:FigEightReg}
				with $\varepsilon = 1/4$. The blow-up procedure pulls the origin apart.}
	\label{fig:FigureEightBlowUp}
\end{figure}
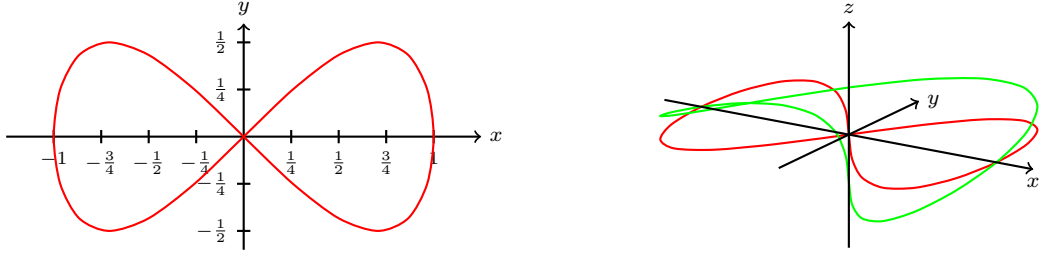

As second example consider the figure-eight curve (see left plot of Figure~\ref{fig:FigureEightBlowUp}) given by
\begin{equation}\label{eqn:FigureEight}
	y^2-x^2+x^4 = 0 \;. 
\end{equation}
The desingularization yields
\begin{equation}\label{eqn:FigEightReg}
	\begin{aligned}
		\left(\frac{z}{\varepsilon}\right)^2 + x^2 &= 1 \;,\\
		y - \frac{z}{\varepsilon} x &= 0 \;.
	\end{aligned}
\end{equation}
A parametrization $\tilde{\gamma}_{\varepsilon}: [-\pi,\pi) \rightarrow \mathbb{R}^3$ is given $ \tilde{\gamma}_{\varepsilon} = (\sin(\theta) , \cos(\theta)\sin(\theta) , \varepsilon \cos(\theta))$.
It is clear from this parametrization that the $z$-axis intersects with the regularized curve $\tilde{S}_{\varepsilon}$ twice.
Indeed, the blow-up procedure
regularized the curve by pulling it apart at the origin as shown in the right plot of Figure~\ref{fig:FigureEightBlowUp}. 

This situation is very different from the cardioid because here it depends on the boundary conditions if we obtain a consistent regularization or not.
The original reaction-diffusion problem consists of a PDE on the figure-eight curve excluding the origin, together with boundary conditions at the origin.
If we impose boundary conditions such that the origin is a cross-junction then resolution of the singularity changes the topology 
since the green curve in the right plot of Figure~\ref{fig:FigureEightBlowUp} exhibits a different connectivity. 
In this case considering the same reaction-diffusion PDE on $\tilde{S}_{\varepsilon}$ will not give a consistent regularization of the problem.
If we are given boundary conditions such that the connectivity matches that of the desingularization, then the same reaction-diffusion problem on $\tilde{S}_{\varepsilon}$ gives us the desired regularization of the problem.

A sufficient condition for a consistent regularization is that the blow-up procedure does not change the topology of the surface.
We conjecture---since there is evidence in our examples---that this condition is satisfied if for every singular point $\vec{p}$ there is a hyperplane passing through $\vec{p}$ which does not intersect the surface at a different point, i.e., the surface is contained in a half-space defined by this hyperplane through $\vec{p}$.

\subsection{Surfaces with an Isolated Singularity}\label{sect:SingSurf}
\begin{figure}[tbp]
	\begin{minipage}{.475\textwidth}
		\centering
		\includegraphics[width=.9\textwidth]{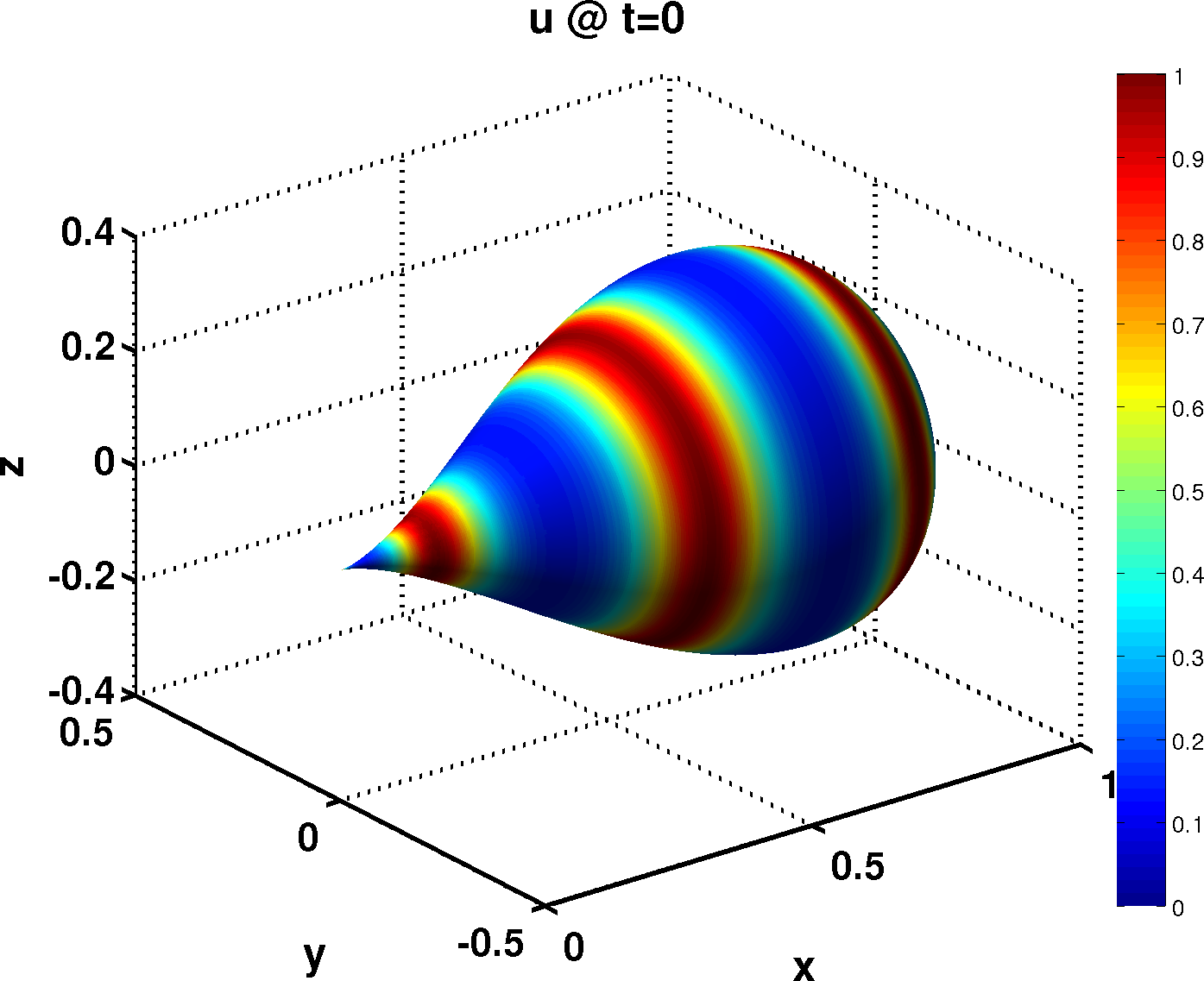}
	\end{minipage}
	\hfill
	\begin{minipage}{.475\textwidth}
		\centering
		\includegraphics[width=.9\textwidth]{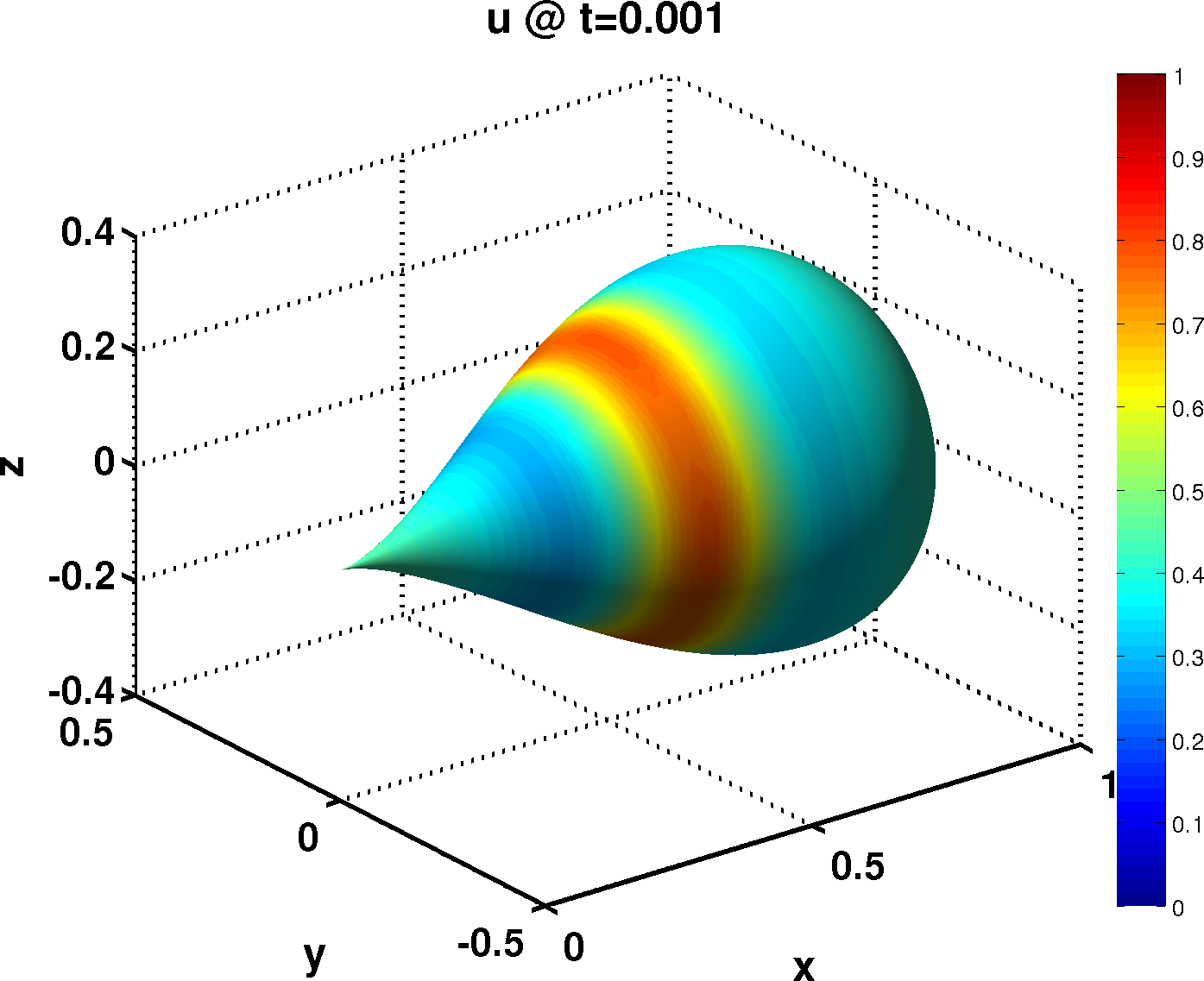}
	\end{minipage}
	\caption{Reaction-diffusion on the $\varepsilon=0.5$-blow-up of the singular surface $S$ given by $y^2+z^2 = x^3-x^4$.
	The images show the pull-down $u_{\varepsilon}$ of the solution $v_{\varepsilon}$ of the IVP \eqref{eqn:RegZeckPDE} onto the singular surface $S$.
	The left image shows the initial state $u_{\varepsilon}$ at time $0$ while the right image shows $u_{\varepsilon}$ at time $t_f=0.001$.} \label{fig:Zeck}
\end{figure}
So far, we have considered curves with a cusp singularity, however, we are also interested in applying our
strategy to surfaces with similar singularities. Consider the singular surface $S$ given by $y^2+z^2 = x^3-x^4$ (see Figure~\ref{fig:Zeck}).
This is the surface of revolution derived from our familiar curve \eqref{eqn:CuspCurve}.
In order to desingularize $S$ we follow the same steps as in Section~\ref{sect:SingularityResolution}.
The only difference is that a description of all straight lines that pass through the origin requires now two equations since $S$ is embedded in $\R^3$. 
Similar to the case of our cuspidal curve, we obtain a one-parameter family of desingularized surfaces $\tilde{S}_{\varepsilon}$ which are embedded in $\R^5$.
These are given by the system 
\begin{equation}\label{eqn:SysBlowUpSurf}
	\begin{aligned}
		\left(\frac{\xi}{\varepsilon}\right)^2 + \bigg( \frac{\eta}{\varepsilon}\bigg)^2 + \left( x - \frac{1}{2}\right)^2 - \frac{1}{4} &=0 \;, \\
		\varepsilon y - \xi x &= 0 \;,\\
		\varepsilon z - \eta x &= 0 \;.
	\end{aligned}
\end{equation}
We remark that all surfaces $\tilde{S}_{\varepsilon}$ are linear deformations of $\tilde{S}=\tilde{S}_{1}$,
and that the deformation matrix is $\Defo_{\varepsilon} = \diag\{1,1,1,\varepsilon,\varepsilon \}$.

The numerical results displayed in Figure~\ref{fig:Zeck} are for the reaction-diffusion problem
\begin{align}
	& \pd_t v_{\varepsilon} = \laplace_{\tilde{S}_{\varepsilon}} v_{\varepsilon} - v_{\varepsilon} \;, \qquad t \in (0,\infty) \;,\; \x \in \tilde{S}_{\varepsilon} 
	\;, \qquad v_{\varepsilon}(0,\x) = u_0(\x)\;, \label{eqn:RegZeckPDE}  
\end{align}
where $\tilde{S}_{\varepsilon}$ is the regular two dimensional surface embedded in $\R^5$ given by \eqref{eqn:SysBlowUpSurf} with $\varepsilon=1/2$
and $\x := (x,y,z,\xi,\eta)$.
The initial condition is
\begin{align}
	& u_0(\x) = \frac{\exp \left( 2 \cdot \cos(4 \arccos (2x-1))^2 \right)}{7.5}  \;. \label{eqn:RegZeckIC}
\end{align}

To obtain a domain amenable to numerical computation we transform, as in Section~\ref{sect:RegTrafo}, via $w_{\varepsilon}( \x) = v_{\varepsilon} (\Defo_{\varepsilon} \x)$ 
to an equivalent problem on $\tilde{S}$. This yields a new domain independent of $\varepsilon$ and the variable coefficient PDE problem:  
\begin{align}
	& \pd_t w_{\varepsilon} = \trace \left( C_{\varepsilon}^T \; D_{\tilde{S}} \left(C_{\varepsilon} \; \nabla_{\tilde{S}} w_{\varepsilon}\right) \right) - w_{\varepsilon} 
	\;, \qquad t \in (0,\infty) \;,\; \x \in \tilde{S}
	\;, \qquad w_{\varepsilon}(0,\x) = u_0(\x) \;. \label{eqn:RegZeckPDETrafo} 
\end{align}
This transformation is a generalization of Lemma~\ref{lem:DiffeoPDESwitch} to surfaces of arbitrary dimension. 
The operators $D_{\tilde{S}}$ and $\nabla_{\tilde{S}}$ denote the Jacobian and the gradient intrinsic to $\tilde{S}$ and  
the coefficient matrix $C_{\varepsilon}$ is given by the expression
\begin{align}
	& C_{\varepsilon} =  \Defo_{\varepsilon} \Proj_{\tilde{S}} \left( \Defo_{\varepsilon}^{-2} - \Defo_{\varepsilon}^{-2} N \; (N^T \Defo_{\varepsilon}^{-2} N)^{-1} N^T \Defo_{\varepsilon}^{-2} \right) \;.
\end{align}
The matrix $N$ is the transpose of the Jacobian of system \eqref{eqn:SysBlowUpSurf}, and its columns span the normal space.

We solved the regularized problem \eqref{eqn:RegZeckPDETrafo} using the explicit closest point iteration (Algorithm~\ref{algo:expl}). 
Our computations were performed on a banded grid around the surface $\tilde{S} \subset \R^5$.
The banded grid is a subset of the virtual $21 \times 21 \times  21 \times  21 \times  21$ Cartesian grid defined on the box
$[-0.3,1.2] \times  [-0.75 , 0.75] \times  [-0.75 , 0.75] \times  [-0.75 , 0.75] \times  [-0.75 , 0.75]$
and contains 0.05 \% of the points of full virtual grid.
Figure~\ref{fig:Zeck} shows $u_{\varepsilon}$ at times $t=0$ and $t=0.001$, where $u_{\varepsilon}$ denotes the pull-down of $v_{\varepsilon}$ onto the singular surface $S$:
\begin{equation}\label{eqn:ZeckPullDown}
	u_{\varepsilon}(t,x,y,z) = v_{\varepsilon} \left(t,x,y,z, \varepsilon \frac{y}{x} , \varepsilon \frac{z}{x} \right) = w_{\varepsilon} \left(t,x,y,z, \frac{y}{x} , \frac{z}{x} \right)\;.
\end{equation}
This higher dimensional numerical result illustrates that the ideas presented are applicable to certain algebraic surfaces of arbitrary dimension and codimension where the surface has isolated singular points.
Of course there are limitations regarding the type of singularity, for example, a singular point must be such that its blow-up is a smooth surface in $\R^n$ which has the same
topology as the original singular surface (cf. Section~\ref{sect:BlowUpGenerality}). 

This example also illustrates that the desingularized surface is embedded in a high dimensional space in general.
This side effect leads us to pay close attention to numerical efficiency. For example, a computational band
should be used that is optimal or nearly optimal; see e.g. \cite[Appendix A]{Colin1}.   

Recall that for the problem of the cuspidal curve we proved in Section~\ref{sect:ConvTheo} that the pull-down $u_{\varepsilon}$ 
converges to the solution $u$ on the original cuspidal curve without the singular point. Such a convergence investigation for surfaces would be an
interesting subject for future work. Note, however, that a different analysis may be required due to the inherently more complex geometry of two dimensional surfaces
over one dimensional curves.  
Assuming convergence of $u_{\varepsilon}$ given by \eqref{eqn:ZeckPullDown} as $\varepsilon \to 0$, we would guess that the limit $u$ is a solution of the problem
\begin{align}
	& \pd_t u = \laplace_{S} u - u \;, \qquad t \in (0,\infty) \;,\; \x = (x,y,z) \in S \wo \{0\} \;, \qquad u(0,\x) = u_0(\x)  \;, 
\end{align}
on the domain $S \wo \{0\}$ where $S = \{ (x,y,z) \in \R^3 : y^2 + z^2 = x^3-x^4\}$.

\subsection{Curves and Surfaces with Multiple Singularities}
\begin{figure}[tbp]
	\centering
	\begin{tikzpicture}[scale=4,thick]
		\draw[->] (-.25,0) -- (1.25,0) node[right] {$x$};
		\foreach \x/\xtext in {.25/\frac{1}{4},.5/\frac{1}{2},.75/\frac{3}{4},1} \draw (\x ,1pt) -- (\x ,-1pt) node[anchor=north] {\scriptsize $\xtext$};

		\draw[->] (0,-.25) -- (0,.25) node[above] {$y$};
		\foreach \y/\ytext in {-.125/-\frac{1}{8}, 0, .125/\frac{1}{8}} \draw (1pt, \y) -- (-1pt,\y) node[anchor=east] {\scriptsize $\ytext$};

		\draw[color=red,domain=-3.141:3.141,smooth,variable=\t] plot ({.5*(1+cos(\t r))},{-.125*sin(\t r)^3});
	\end{tikzpicture}
	\caption{The red curve is given by $y^2=x^3 (1-x)^3$ and possesses two isolated singularities.} \label{fig:DoubleCusp}
\end{figure}
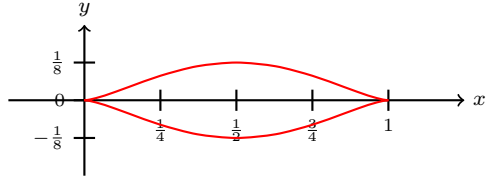
When finitely many isolated singularities arise, our procedure is to resolve one singularity at a time.
We illustrate this generalization by resolving the two isolated singularities of the curve shown in Figure~\ref{fig:DoubleCusp} which is given by 
\begin{equation}
	y^2 - x^3 (1-x)^3 = 0 \;.
\end{equation}
The singular points are the origin $(x,y)=(0,0)$ and the point $(x,y)=(1,0)$.
Starting out with the point $(x,y)=(1,0)$, we adjoin an equation describing the straight lines passing through $(1,0)$. This yields
\begin{equation}
	\begin{aligned}
		y^2 - x^3 (1-x)^3 &= 0 \;, \label{eqn:dCusp}\\
		y - \frac{z}{\varepsilon_1} (x-1) &= 0 \;.
	\end{aligned}
\end{equation}
The next steps agree exactly with Section~\ref{sect:SingularityResolution}. Specifically, we replace $y$ in the first equation of \eqref{eqn:dCusp} with $(x-1) z/\varepsilon$, we factor out $(1-x)^2$, and we
desingularize by neglecting this factor. This yields
\begin{equation}\label{eqn:ResFirst}
	\begin{aligned}
		\left(\frac{z}{\varepsilon_1}\right)^2 - x^3 (1-x) &= 0 \;,\\
		y - \frac{z}{\varepsilon_1} (x-1) &= 0 \;.
	\end{aligned}
\end{equation}
In the new system \eqref{eqn:ResFirst}, only the first equation is singular. Indeed, it is exactly our example from Section~\ref{sect:SingularityResolution}. 
We therefore adjoin an equation for the straight lines passing through the origin $(x,z)=(0,0)$ as in Section~\ref{sect:SingularityResolution}.
This yields the desingularized system
\begin{equation}
	\begin{aligned}
		\left(\frac{w}{\varepsilon_1 \varepsilon_2}\right)^2 + \left(x-\frac{1}{2}\right)^2 &= \frac{1}{4} \;,\\
		y - \frac{z}{\varepsilon_1} (x-1) &= 0 \;,\\
		z - \frac{w}{\varepsilon_2} x &= 0\;.
	\end{aligned}
\end{equation}
This system defines a curve in four dimensional space, $(x,y,z,w) \in \R^4$, which can be parametrized by 
$\gamma(\theta) = ( (1+\cos(\theta))/2 , -\sin(\theta)^3/8, \varepsilon_1 \sin(\theta) (1+\cos(\theta))/4, \varepsilon_1 \varepsilon_2 \sin(\theta)/2 )$.
The first two components of $\gamma$ give a parametrization of the original singular curve.
Formally, we can introduce a new parameter $\varepsilon_i$ for each singularity. But looking at the deformation matrix
$\Defo_{\varepsilon_1,\varepsilon_2} = \diag\{1,1,\varepsilon_1,\varepsilon_1 \varepsilon_2 \}$ we observe that it is sufficient to keep only parameter $\varepsilon_1 = \varepsilon$
and set the other equal to $\varepsilon_2 = 1$. The deformation matrix $\Defo_{\varepsilon} = \diag\{1,1,\varepsilon,\varepsilon\}$ will still give us the projection in the case of $\varepsilon=0$.
Accordingly, the desingularized system is
\begin{equation}
	\begin{aligned}
		\left(\frac{w}{\varepsilon}\right)^2 + \left(x-\frac{1}{2}\right)^2 &= \frac{1}{4} \;,\\
		y - \frac{z}{\varepsilon} (x-1) &= 0 \;,\\
		z - w x &= 0\;.
	\end{aligned}
\end{equation}

Finally, we remark that the dimension of the embedding space increases only by one per resolved singularity. After the resolution of the first singularity
we obtained system \eqref{eqn:ResFirst} in $\R^3$, so one might have expected that for the resolution of the second singularity we had to adjoin two equations  
for the description of a straight line in $\R^3$ (cf. Section~\ref{sect:SingSurf}). That would have yielded a system of four equations in five variables.
Interestingly, the resolution of a singularity does not change the number of variables in the first equation. Thus, every single resolution operation can be performed by
considering only the first equation in two dimensional space.

Analogously, for a two dimensional algebraic surface in $\R^3$ with several isolated singularities we adjoin two equations each time we resolve a singularity.
For example, consider the surface obtained from revolving the curve shown in Figure~\ref{fig:DoubleCusp} around the $x$-axis:
\begin{equation}
	y^2 + z^2 - x^3 (1-x)^3 = 0 \;.
\end{equation}
The resolution of the singularities (one singularity at a time) gives the desingularized system
\begin{equation}
	\begin{aligned}
		\left(\frac{a}{\varepsilon}\right)^2 + \left(\frac{b}{\varepsilon}\right)^2 + \left(x-\frac{1}{2}\right)^2 &= \frac{1}{4} \;,\\
		y - \frac{a}{\varepsilon} (x-1) &= 0 \;,\\
		z - \frac{b}{\varepsilon} (x-1) &= 0 \;,\\
		a - c x &= 0\;,\\
		b - d x &= 0\;.
	\end{aligned}
\end{equation}
This is a system of five equations involving seven variables $(x,y,z,a,b,c,d) \in \R^7$ as opposed to a system of seven equations in nine variables. 

\section{Conclusion}
This paper presents a framework and numerical embedding method for evolving reaction-diffusion equations on algebraic surfaces with isolated singularities.
We explored and proved several analytical results in both geometry and geometric PDEs in Section~\ref{sect:Regularization}. 
In particular, we used the blow-up technique from algebraic geometry to resolve singularities in a domain,
proved the convergence of the solutions for the reaction-diffusion equation on $\tilde{S}_{\varepsilon}$ to the solution of the reaction-diffusion equation on $S$,
and constructed a variable-coefficient PDE on $\tilde{S}$ that corresponds to a PDE posed on the domain $\tilde{S}_{\varepsilon}$. 
Together these results provide a solid mathematical foundation from which we built a numerical technique. 

In constructing our numerical method, we first applied Lemma~\ref{lem:DiffeoPDESwitch}.
This lemma allows us to pose a variable-coefficient PDE on a smooth and numerically well-resolved curve $\tilde{S}$ that corresponds to the reaction-diffusion PDE on $\tilde{S}_{\varepsilon}$.
We then applied the Closest Point Method to evolve the variable-coefficient reaction-diffusion equation. 
Our numerical results are stable and demonstrate that the method obtains second order accuracy in space. 

In Section~\ref{sect:Extensions} we discussed extensions as well as limitations of our approach. In particular, we saw that preservation of the topology plays an important role.
Notably, the convergence theory of Section~\ref{sect:ConvTheo} applies to all closed curves with finitely many singularities as long as the blow-up procedure preserves the topology of the curve.
This is because the preservation of the topology implies that \eqref{eqn:ALdiff} in the proof of Lemma~\ref{lem:FourierConvergence} holds (with a different constant).

As desingularization of the PDE-domain proved to be useful for regularization we are interested in combining embedding techniques such as the Closest Point Method
with desingularization other than blow-up. This might help to reduce some of the limitations
of our method that have been identified in this paper. For example, the problem of topological change with the cardioid as discussed in Section~\ref{sect:BlowUpGenerality} 
is resolved if we use the curve parametrized by $\tilde{\gamma}_{\varepsilon} = ((1+\cos(\theta))\cos(\theta)/2 , (1+\cos(\theta)) \sin(\theta)/2, \varepsilon \sin(\theta))$.
This curve is suitable for regularizing the cardioid since it corresponds to the cardioid by projection, is smooth, and has the same topology as the cardioid.
Besides further extending our study of surfaces, the investigation of other desingularization techniques will be an interesting subject for future work.

\begin{acknowledgements}
	The authors thank Colin Macdonald for many helpful discussions. 
	The authors thank also all the people who contributed to the \verb|cp_matrices| code 
	(see \url{github.com/cbm755/cp_matrices}), in particular Colin Macdonald, Ingrid von Glehn, and Yujia Chen. 
\end{acknowledgements}

\bibliographystyle{spmpsci}      
\bibliography{FoCMBlowUp}			

\end{document}